\documentclass{amsart}
\usepackage{amssymb, amsmath, amscd, mathrsfs,marvosym, psfrag,color} 

\input xy
\xyoption{all}
\DeclareMathAlphabet{\mathpzc}{OT1}{pzc}{m}{it}

\newtheorem{theorem}{Theorem}[section]
\newtheorem*{MainTheorem}{Main Theorem}

\newtheorem{proposition}[theorem]{Proposition}

\newtheorem{lemma}[theorem]{Lemma}

\theoremstyle{definition}
\newtheorem{definition}[theorem]{Definition}

\theoremstyle{remark}
\newtheorem{remark}[theorem]{Remark}

\newcommand{\CO}{{\mathcal O}}

\newcommand{\CX}{{\mathcal X}}

\newcommand{\SA}{{\mathscr A}}
\newcommand{\SB}{{\mathscr B}}

\newcommand{\SF}{{\mathscr F}}

\newcommand{\SK}{{\mathscr K}}

\newcommand{\SR}{{\mathscr R}}

\newcommand{\SZ}{{\mathscr Z}}

\newcommand{\fp}{{{\mathfrak p}}}

\newcommand{\hE}{{\widehat E}}
\newcommand{\hF}{{\widehat F}}

\newcommand{\hM}{{\widehat M}}

\newcommand{\hf}{{\widehat f}}

\newcommand{\tM}{{\widetilde{M}}}

\newcommand{\tf}{{\widetilde{f}}}

\newcommand{\DZ}{{\mathbb Z}}

\newcommand{\DN}{{\mathbb N}}
\newcommand{\DQ}{{\mathbb Q}}

\newcommand{\DF}{{\mathbb F}}

\newcommand{\End}{{\operatorname{End}}}
\newcommand{\Ext}{{\operatorname{Ext}}}
\newcommand{\Hom}{{\operatorname{Hom}}}

\newcommand{\im}{{\operatorname{im}\,}}

\newcommand{\ol}{\overline}

\newcommand{\id}{{\operatorname{id}}}

\newcommand{\Quot}{{\operatorname{Quot\,}}}

\newcommand{\irr}{{\operatorname{irr}}}

\newcommand{\comment}[1]{}

\newcommand{\lgl}{\langle}
\newcommand{\rgl}{\rangle}

\newcommand{\qchoose}[2]{\left[{#1\atop#2}\right]}

\begin{document}

\pagenumbering{arabic}
\title[]{Tilting modules and torsion phenomena} \author[]{Peter Fiebig}
\begin{abstract}  We construct  families of representations for quantum groups over $\DZ[v,v^{-1}]$-algebras that interpolate between Weyl modules and tilting modules. These families might be  candidates for objects with characters satisfying the {\em generations of characters} philosophy of Lusztig and Lusztig-Williamson. 
\end{abstract}

\address{Department Mathematik, FAU Erlangen--N\"urnberg, Cauerstra\ss e 11, 91058 Erlangen}
\email{fiebig@math.fau.de}
\maketitle

\section{Introduction}
Let $R$ be a root system and $\SA$ a  unital $\SZ:=\DZ[v,v^{-1}]$-algebra. To these data one associates the quantum group $U_\SA=U_\SZ\otimes_\SZ \SA$, where $U_\SZ$ is Lusztig's integral quantum group that is defined using divided powers. Then $U_\SZ$, and hence $U_\SA$, admits a triangular tensor decomposition, and so it gives rise to a version $\CO_\SA$ of the classical BGG-category $\CO$ associated with a complex semisimple Lie algebra. 

In this article we mostly assume that $\SA$ is local, factorial and {\em generic} (meaning that the quantum integers $[n]_d$ with $n,d\ne 0$ do not vanish in $\SA$). We consider the following two families 
 of objects in $\CO_\SA$. The first is the family of Weyl modules $W_\SA(\lambda)$, the second the family of tiliting modules $T_\SA(\lambda)$. Both families are parametrized by the highest weight $\lambda$, which is a dominant and integral weight. Note that if one is interested in non-generic situations (e.g., $\SA$ being a field of positive characteristic or a cyclotomic extension of $\DQ$), one can obtain valuable information by deforming objects such as Weyl or tilting modules to a generic and local algebra. 
  
There are two, quite different, notions of tilting modules in $\CO_\SA$. From one perspective, it is reasonable to consider the Verma modules in $\CO_\SA$ as the standard objects. Then a tilting module is an object in $\CO_\SA$ that admits a Verma filtration, as well as a dual Verma filtration. As each Verma module is free over $\SA$ of infinite rank, we obtain objects of infinite rank in $\CO_\SA$. One might call these tilting modules the {\em fat} tilting modules.  On the other hand, if one is interested in the subcategory of $\CO_\SA$ that contains objects of finite rank, then one should consider the Weyl modules as the standard objects. A tilting module in this context is an object  that admits a Weyl filtration and a dual Weyl filtration. For any dominant $\lambda$ such a tilting module exists and is unique up to isomorphism, and we might call these the {\em thin} tilting modules. In this article, by ``tilting module'' we mean thin tilting modules. 

A particularly important problem in representation theory is the determination of the Weyl multiplicities for the thin tilting modules. If $\SA$ is a field of characteristic zero and the image $q$ of $v$ in $\SA$ is a root of unity $\ne 1$ (the ``quantum case''), then these are determined by Soergel in \cite{Soe}. If $\SA$ is a field of positive characteristic and $q=1$ (the ``modular case''), then one can calculate these multiplicities in the diagrammatic Hecke category of Elias and Williamson (cf.~\cite{AMRW,BR,C}). So far a combinatorial formula for the results of these calculations, i.e.~for the $p$-canonical basis in the Hecke algebra, is not found nor even conjectured. 

In the modular case, Lusztig and Williamson conjectured in \cite{LW} the existence of {\em generations of tilting characters}, building on an analogous idea of Lusztig \cite{L1} in the case of irreducible characters. There should be a set $\{\theta_\lambda^l\}$ of elements in the group algebra of the weight lattice, indexed by a dominant $\lambda$ and an integer $l\ge0$, satisfying certain properties. Among those properties are the following: $\theta_\lambda^0$ is the character of the Weyl module, $\theta_\lambda^1$ is the character of the tilting module over the quantum group at a $p$-th root of unity, and $\theta_\lambda^\infty$ is the character of the modular tilting module. Moreover, each $\theta_\lambda^l$ should be a positive linear combination of $\theta_\mu^{l-1}$'s. A formula or an algorithm for these characters is not known so far. It is natural to assume that the $\theta_\lambda^l$ are tilting characters for a new type of ``quantum groups''.  The ``Morava $E$-theoretical quantum groups'' that were recently constructed by Yang and Zhao in \cite{YZ} are natural candidates for these algebraic structures. 

In this paper we propose a different approach. Let us fix an odd prime number $p$ and assume that $p\ne 3$ if the root system contains a component of type $G_2$. We denote by $\SZ_{\fp}$ the ``quantum $p$-adic integers'', i.e.~the localization of $\DZ[v]$ at the kernel of the homomorphism $\DZ[v]\to\DF_p$ that sends $v$ to $1$.  We fix a local and generic $\SZ_\fp$-algebra $\SA$ (the case $\SA=\SZ_\fp$ being the main example),  and construct a  family of objects $T_{\Theta}(\lambda)$ in $\CO_\SA$ that interpolate between the Weyl modules and the tilting modules.  Here, $\Theta$ denotes a subset of the natural numbers $\{0,1,2,\dots\}$,   and $\lambda$ again is a dominant weight (note that our construction works for arbitrary weights $\lambda$).  This family of objects satisfies the following properties.

\begin{MainTheorem}  Let $\lambda$ be a dominant weight.
\begin{enumerate}
\item $T_{\emptyset}(\lambda)$ is the Weyl module with highest weight $\lambda$. 
\item $T_{\DN}(\lambda)$ is the indecomposable tilting module with highest weight $\lambda$.
\item Each $T_{\Theta}(\lambda)$ admits a Weyl filtration.
\item  If $l\in\Theta$, then the character of $T_{\Theta}(\lambda)$ is a sum of tilting characters of the quantum group at an $p^l$-th root of unity.
\end{enumerate}
\end{MainTheorem}

The main idea for the construction is the following. For any object $M$ in $\CO_\SA$ that admits a Weyl filtration and any weight $\mu$ we consider a certain torsion $\SA$-module (denoted by $M_{\mu,\max}/M_{\lgl\mu\rgl}$ in the article). We show that for the Weyl modules, these torsion modules are as large as possible, whereas for the tilting modules they vanish. We also show that  these torsion modules are annihilated by a product of the $p^l$-th cyclotomic polynomials $\tau_l\in\SZ_{\fp}$. The objects $T_{\Theta}(\lambda)$ are now the minimal objects that have the property that they do not contain $\tau_{l}$-torsion for all  $l\in\Theta$.

The family of characters $\gamma_\lambda^l$ of $T_{\{1,\dots,l\}}(\lambda)$ might be a candidate for the conjectured family $\{\theta_\lambda^l\}$. Note that it is not clear from the construction whether each $\gamma_\lambda^l$ can be written as a sum of $\gamma_\mu^{l-1}$'s. The above theorem only establishes that each $\gamma_\lambda^l$ can be written as a sum of $\gamma_\mu^0$'s, or as a sum of $\gamma_\mu^1$,s.

As a main tool for the construction of the above objects we  introduce an auxiliary category $\CX_\SA$ that contains ``$X$-graded objects with operators in simple root directions''. The definition of $\CX_\SA$ is strongly motivated by the ideas in the article \cite{LefOp}. We show that $\CX_\SA$ can be fully embedded into category $\CO_\SA$, and  that the subcategory $\CX_\SA^f$ that contains the objects  that are free of finite rank over $\SA$, coincides with the subcategory of $\CO_\SA$ of objects that admit a finite Weyl filtration.

{\bf Acknowledgements:}  I would like to thank Henning Haahr Andersen for valuable comments on an earlier version of the article. The article is partly based upon work supported by a grant of the Institute for Advanced Study School of Mathematics.

\section{$X$-graded spaces with operators}\label{sec-Xgradop}\label{sec-Xgrad}
We fix a root system $R$ and define the  category $\CX=\CX_\SA(R)$ of {\em graded spaces with operators}. This is a linear category over a unital $\DZ[v,v^{-1}]$-algebra $\SA$ that is factorial.  In  Section \ref{sec-qg} we show that $\CX$ can be embedded into the category of representations of the quantum group $U_\SA$  associated with $R$ over $\SA$. 

\subsection{Quantum integers}\label{subsec-qint} Let $v$ be an indeterminate and set $\SZ:=\DZ[v,v^{-1}]$. For $n\in\DZ$ and $d>0$  we define the quantum integer 
$$
[n]_d:=\frac{v^{dn}-v^{-dn}}{v^d-v^{-d}}=\begin{cases}
0,&\text{ if $n=0$},\\
v^{d(n-1)}+v^{d(n-3)}+\dots+v^{d(-n+1)},&\text{ if $n> 0$}, \\
-v^{d(-n-1)}-v^{d(-n-3)}-\dots-v^{d(n+1)},&\text{ if $n<0$}.
\end{cases}
$$
The quantum factorials are given by   $[0]_d^!:=1$ and
$
[n]_d^!:=[1]_d\cdot[2]_d\cdots[n]_d
$ for $n\ge 1$. The quantum binomial coefficients are  $\qchoose{n}{0}_d:=1$ 
and 
$
\qchoose{n}{r}_d:=\frac{[n]_d\cdot[n-1]_d\cdots[n-r+1]_d}{[1]_d\cdot [2]_d\cdots [r]_d}$ for $n\in\DZ$ and $r\ge 1$.
Note that under the ring homomorphism $\SZ\to\DZ$ that sends $v$ to $1$, the quantum integer $[n]_d$ is sent to $n$ for all $n\in\DZ$, independently of $d$. Hence $[n]^!_d$ is sent to $n!$ and $\qchoose{n}{r}_d$ to $n\choose r$.

\subsection{Graded spaces with operators}\label{subsec-gradspac}
For the rest of this article we fix a root system $ R$ in a real vector space $V$ and a basis $\Pi$ of $R$.  The coroot for $\alpha\in R$ is $\alpha^\vee\in V^\ast$, and the weight lattice is $X:=\{\lambda\in V\mid \lgl \lambda,\alpha^\vee\rgl\in\DZ\text{ for all $\alpha\in R$}\}$. We
denote by $\le$ the standard partial order on $X$, i.e.~$\mu\le\lambda$ if and only if $\lambda-\mu$ can be written as a sum of elements in $\Pi$.

\begin{definition} \begin{enumerate}
\item A subset $I$ of $X$ is called {\em closed} if $\mu\in I$ and $\mu\le\lambda$ imply $\lambda\in I$.
\item A subset $S$ of $X$ is called {\em  quasi-bounded} if for any $\mu\in X$ the set $\{\lambda\in S\mid \mu\le\lambda\}$ is finite.
\end{enumerate}
\end{definition}

Now let  $\SA$ be a unital $\SZ$-algebra. A general assumption throughout this article is that $\SA$ is a factorial domain, i.e. it is a domain and every element can be uniquely (up to units) written as a product of irreducible elements. Let $I$ be a closed subset of $X$, and let $M=\bigoplus_{\mu\in I}M_\mu$ be an $I$-graded $\SA$-module. 
We say that $\mu$ is a {\em weight of $M$} if $M_\mu\ne\{0\}$. 
For any $\mu\in I$, $\alpha\in\Pi$, and $n>0$ let
\begin{align*}
F_{\mu,\alpha,n}&\colon M_{\mu+n\alpha}\to M_\mu,\\
E_{\mu,\alpha,n}&\colon M_{\mu}\to M_{\mu+n\alpha}
\end{align*}
be $\SA$-linear homomorphisms. 
It is convenient to set $E_{\mu,\alpha,0}=F_{\mu,\alpha,0}:=\id_{M_\mu}$.
In the following we often suppress the index ``$\mu$'' in the notation of the $E$- and $F$-maps if the source of the maps is clear from the context, but we sometimes also write $E_{\alpha,n}^M$ and $F_{\alpha,n}^M$ to specify the object $M$ on which these homomorphisms are defined.  

Now we list some conditions  on the above data. Denote by $A=(\lgl\alpha,\beta^\vee\rgl)_{\alpha,\beta\in\Pi}$ the Cartan matrix associated with the root system $R$. Then there exists a vector $d=(d_\alpha)_{\alpha\in\Pi}$ with entries in $\{1,2,3\}$ such that $(d_\alpha \lgl\alpha,\beta^\vee\rgl)_{\alpha,\beta\in\Pi}$ is symmetric and such that each irreducible component of $R$ contains some $\alpha\in \Pi$ with  $d_\alpha=1$. The first two conditions are as follows. 

\begin{enumerate}
\item[(X1)] The set of weights of $M$  is  quasi-bounded and each $M_\mu$ is finitely generated as an $\SA$-module. 
\item[(X2)] For all $\mu\in I$, $\alpha,\beta\in\Pi$, $m,n>0$, and $v\in M_{\mu}$,
$$
E_{\alpha,m}F_{\beta,n}(v)=
\begin{cases}
F_{\beta,n}E_{\alpha,m}(v),&\text{ if $\alpha\ne\beta$,}\\
\sum_{0\le r\le \min(m,n)} \qchoose{\lgl\mu,\alpha^\vee\rgl+m-n}{r}_{d_\alpha}F_{\alpha,n-r}E_{\alpha,m-r}(v),&\text{ if $\alpha=\beta$}.
\end{cases}
$$
\end{enumerate}
(The cautious reader may want to have a look at Equation (a2) in Section 6.5 of \cite{L2} to get an idea of where the second equation comes from.)

\subsection{Torsion subquotients}
In order to formulate the third condition, we need some definitions. For any $\mu\in I$ define
$$
M_{\delta\mu}:=\bigoplus_{\alpha\in\Pi, n>0} M_{\mu+n\alpha}.
$$
Let
\begin{align*}
E_\mu&\colon M_\mu\to M_{\delta\mu},\\
F_\mu&\colon M_{\delta\mu}\to M_\mu
\end{align*}
be the column and the row vector with entries $E_{\mu,\alpha,n}$ and $F_{\mu,\alpha,n}$, resp. We sometimes write $E_\mu^M$ and $F_\mu^M$ in order to  specify the object $M$ on which $E_\mu$ and $F_\mu$ act. Set
\begin{align*}
M_{\{\mu\}}&:=E_\mu(\im F_\mu),\\
M_{\lgl\mu\rgl}&:=E_\mu(M_\mu).
\end{align*}
So we have inclusions $M_{\{\mu\}}\subset M_{\lgl\mu\rgl}\subset M_{\delta\mu}$.

\begin{enumerate}
\item[(X3)] For all $\mu\in I$ the following holds:
\begin{enumerate}
\item The restriction of $E_\mu\colon M_\mu\to M_{\delta\mu}$ to $\im F_\mu\subset M_{\mu}$ is injective and hence induces an isomorphism $\im F_\mu\xrightarrow{\sim} M_{\{\mu\}}$.
\item  The quotient $M_{\lgl\mu\rgl}/M_{\{\mu\}}$ is a torsion $\SA$-module. 
\item $M_\mu/\im F_\mu$ is a free $\SA$-module.
\end{enumerate} 
\end{enumerate}
Here is our first, rather easy, result.

\begin{lemma} \label{lemma-torfree} Suppose that our data satisfies (X1) and (X3). Then $M$ is a torsion free $\SA$-module. In particular, the spaces $M_{\delta\mu}$, $M_{\{\mu\}}$ and $M_{\lgl\mu\rgl}$  are torsion free $\SA$-modules for all $\mu\in I$.
\end{lemma}

\begin{proof} If $M$ is not torsion free, then assumption (X1) implies that there is a maximal weight $\mu$ of $M$ such that $M_\mu$ is not torsion free. By the maximality of $\mu$,  the module $M_{\delta\mu}$ is torsion free. Hence so is its submodule $M_{\{\mu\}}$. From (X3a) it follows that $\im F_\mu$ is torsion free. By (X3c) the module  $M_\mu/\im F_\mu$ is free,  so $M_\mu$ must be torsion free and we have a contradiction.  
\end{proof}
In particular, the spaces $M_{\delta\mu}$ and $M_{\{\mu\}}$ are torsion free $\SA$-modules for all $\mu\in I$ if (X1) and (X3) hold.

The following results might shed some light on the assumption (X3). Define
$$
M_{\{\mu\},\max}:=\{m\in M_{\delta\mu}\mid \xi m \in M_{\{\mu\}}\text{ for some $\xi\in\SA$, $\xi\ne 0$}\}.
$$
So this is the preimage of the torsion part of $M_{\delta\mu}/M_{\{\mu\}}$ under the quotient map. 
Suppose that $N$ is a submodule of $M_{\delta\mu}$ that contains $M_{\{\mu\}}$.
Then $N/M_{\{\mu\}}$ is a torsion module if and only if $N\subset M_{\{\mu\},\max}$. In particular, condition (X3b) now reads 
$$
M_{\lgl\mu\rgl}\subset M_{\{\mu\},\max}.
$$

\begin{lemma} \label{lemma-soft} Suppose that the assumption (X3a) holds and let $\mu$ be an element in $I$. Then  the following are equivalent.
\begin{enumerate}
\item $M_\mu=\ker E_\mu\oplus\im F_\mu$.
\item $M_{\{\mu\}}=M_{\lgl\mu\rgl}$.  
\end{enumerate}
\end{lemma}
\begin{proof} If (1) holds, then $M_{\lgl\mu\rgl}=E_\mu(M_\mu)=E_\mu(\im F_\mu)=M_{\{\mu\}}$. Suppose that (2) holds, so $E_\mu(M_\mu)=E_\mu(\im F_\mu)$. For each $m\in M_\mu$ there exists then an element $\tilde m\in \im F_\mu$ such that $E_\mu(m)=E_\mu(\tilde m)$, hence $m-\tilde m\in\ker E_\mu$. So $M_\mu=\ker E_\mu+\im F_\mu$. But condition (X3a) reads $\ker E_\mu\cap\im F_\mu=\{0\}$. Hence $M_\mu=\ker E_\mu\oplus\im F_\mu$. 
\end{proof}

Denote by $\SK$ the quotient field of $\SA$. For an $\SA$-module $N$  let $N_\SK:=N\otimes_\SA\SK$ be the associated $\SK$-module.

\begin{lemma}\label{lemma-X3} Suppose that $M$ satisfies condition (X1). Then condition (X3) is equivalent to the following set of conditions.
\begin{enumerate}
\item $M$ is a torsion free $\SA$-module.
\item  For all $\mu \in I$ we have $(M_\mu)_\SK=(\ker E_\mu)_\SK\oplus (\im F_\mu)_\SK$.
\item  Condition (X3c) holds: $M_\mu/\im F_\mu$ is a free $\SA$-module for all $\mu\in I$.
\end{enumerate}
In particular, if $\SA=\SK$ is a field, then condition (X3) simplifies to $M_\mu=\ker E_\mu\oplus \im F_\mu$ for all $\mu\in I$. 
\end{lemma}
\begin{proof} Suppose that (X3) is satisfied. We have already shown in Lemma \ref{lemma-torfree} that (X1) and (X3) imply that $M$ is torsion free as an $\SA$-module.  Moreover, (X3a) says that $\ker E_\mu\cap \im F_\mu=\{0\}$. Now let $m\in M_\mu$. Then, by (X3b), there exists an element $\xi\in\SA$, $\xi\ne 0$ and $m^\prime\in \im F_\mu$ such that $\xi E_\mu(m)=E_\mu(m^\prime)$. So $\xi m-m^\prime$ is contained in the kernel of $E_\mu$ and we deduce $(M_\mu)_\SK=(\ker E_\mu)_\SK+(\im F_\mu)_\SK$. The last two results say that $(M_\mu)_\SK=(\ker E_\mu)_\SK\oplus (\im F_\mu)_\SK$. Hence (1), (2) and (3) are satisfied. 

Now assume that (1), (2) and (3) hold. As $M$ is torsion free we can view it as a subspace in $M_\SK$. Hence $(M_\mu)_\SK=(\ker E_\mu)_\SK\oplus (\im F_\mu)_\SK$ implies that $E_\mu|_{\im F_\mu}$ is injective, i.e.~(X3a). It also implies that 
$E_\mu(\im F_\mu)_\SK=E_\mu(M_\mu)_\SK$, i.e.~$(M_{\{\mu\}})_\SK=(M_{\lgl\mu\rgl})_\SK$. Hence  the cokernel of the inclusion $M_{\{\mu\}}\subset M_{\lgl\mu\rgl}$  is a torsion module, so (X3b) holds. That (X3c) holds is the assumption (3). 
\end{proof}

\subsection{The category $\CX$} Now we are ready to define our auxiliary category. 
Let $I$ be a closed subset of $X$.
\begin{definition} The category $\CX_{\SA,I}$ is defined as follows. Objects are $I$-graded $\SA$-modules $M=\bigoplus_{\mu\in I}M_\mu$ endowed with $\SA$-linear homomorphisms $F_{\mu,\alpha,n}\colon M_{\mu+n\alpha}\to M_\mu$ and $E_{\mu,\alpha,n}\colon M_\mu\to M_{\mu+n\alpha}$ for all $\mu\in I$, $\alpha\in\Pi$ and $n>0$ such that conditions (X1), (X2) and (X3) are satisfied. A morphism $f\colon M\to N$  in $\CX_{\SA,I}$ is a collection of $\SA$-linear homomorphisms $f_\mu\colon M_\mu\to N_\mu$ for all $\mu\in I$ such that the diagrams

\centerline{
\xymatrix{
M_{\mu+n\alpha}\ar[r]^{f_{\mu+n\alpha}}\ar[d]_{F^M_{\alpha,n}}&N_{\mu+n\alpha}\ar[d]^{F^N_{\alpha,n}}\\
M_\mu\ar[r]^{f_\mu}&N_\mu
}
\quad\quad
\xymatrix{
M_{\mu+n\alpha}\ar[r]^{f_{\mu+n\alpha}}&N_{\mu+n\alpha}\\
M_{\mu}\ar[u]^{E^M_{\alpha,n}}\ar[r]^{f_{\mu+n\alpha}}&N_{\mu}\ar[u]_{E^N_{\alpha,n}}
}
}
\noindent
commute for all $\mu\in I$, $\alpha\in\Pi$ and $n>0$.
\end{definition}
If the ground ring is determined from the context, we write $\CX_I$ instead of $\CX_{\SA,I}$. We also write $\CX$ or $\CX_\SA$ for the ``global''  category $\CX_{\SA,X}$. 
 
\begin{remark}\label{rem-mordelta}
If $M$ and $N$ are objects in $\CX_I$ and $f=\{f_\mu\colon M_\mu\to N_\mu\}_{\mu\in I}$ is a collection of homomorphisms, then we denote  by $f_{\delta\mu}\colon M_{\delta\mu}\to N_{\delta\mu}$ the diagonal matrix with entries $f_{\mu+n\alpha}$. Then $f$ is a morphism in $\CX_I$ if and only if for all $\mu\in I$ the diagrams

\centerline{
\xymatrix{
M_{\delta\mu}\ar[d]_{F^M_\mu}\ar[r]^{f_{\delta\mu}}&N_{\delta\mu}\ar[d]^{F^N_\mu}\\
M_\mu\ar[r]^{f_\mu}& N_\mu
}
\quad\quad
\xymatrix{
M_{\delta\mu}\ar[r]^{f_{\delta\mu}}&N_{\delta\mu}\\
M_\mu\ar[u]^{E^M_\mu}\ar[r]^{f_\mu}& N_\mu\ar[u]_{E^N_\mu}
}
}
\noindent commute.
\end{remark}

\section{Extending morphisms}
We retain the notations of the previous section. 
Let $I^\prime\subset I$ be closed subsets of $X$ and let $M$ be an object in $\CX_I$. We define $M_{I^\prime}:=\bigoplus_{\mu\in I^\prime} M_\mu$ and endow it with the homomorphisms $E_{\mu,\alpha,n}$ and $F_{\mu,\alpha,n}$ for all $\mu\in I^\prime$. Then one easily checks that the properties (X1), (X2) and (X3) are preserved, so this defines an object $M_{I^\prime}$  in $\CX_{I^\prime}$. For a morphism $f\colon M\to N$ we obtain a morphism $f_{I^\prime}\colon M_{I^\prime}\to N_{I^\prime}$ by restriction, and this yields a functor
$$
(\cdot)_{I^\prime}\colon \CX_I\to \CX_{I^\prime}
$$
that we call a {\em restriction functor}.

\subsection{The existence of extensions of morphisms}

The following proposition is a cornerstone of the approach outlined in this article. Its proof is not difficult, but lengthy.  

\begin{proposition} \label{prop-mainext}  Let $I^\prime$ be a closed subset of $X$ and suppose that $\mu\not\in I^\prime$ is such that $I:=I^\prime\cup\{\mu\}$ is also closed. Let $M$ and $N$ be objects in $\CX_I$, and let $f^\prime\colon M_{I^\prime}\to N_{I^\prime}$ be a morphism in $\CX_{I^\prime}$.
\begin{enumerate}
\item There exists a unique $\SA$-linear homomorphism $\tilde f_\mu\colon \im F^M_\mu\to N_\mu$ such that the diagrams 

\centerline{
\xymatrix{
M_{\delta\mu}\ar[d]_{F_\mu^M}\ar[r]^{f^\prime_{\delta\mu}}&N_{\delta\mu}\ar[d]^{F_\mu^N}\\
\im F_\mu^M\ar[r]^{\tf_\mu}& N_\mu
}
\quad\quad
\xymatrix{
M_{\delta\mu}\ar[r]^{f^\prime_{\delta\mu}}&N_{\delta\mu}\\
\im F_\mu^M\ar[u]^{E_\mu^M}\ar[r]^{\tf_\mu}& N_\mu\ar[u]_{E_\mu^N}
}
}
\noindent commute. In particular, $f^\prime_{\delta\mu}$ maps $M_{\{\mu\}}$ into $N_{\{\mu\}}$.
\item The following are equivalent.
\begin{enumerate}
\item  There exists a morphism $f\colon M\to N$ in $\CX_I$ such that $f_{I^\prime}=f^\prime$.
\item The homomorphism $f^\prime_{\delta\mu}\colon M_{\delta\mu}\to N_{\delta\mu}$ maps $M_{\lgl\mu\rgl}$ into $N_{\lgl\mu\rgl}$.
\end{enumerate}
\end{enumerate}
\end{proposition}
\begin{proof} First we prove part (1). 
Set $\hM_\mu:=\bigoplus_{\alpha\in\Pi,n>0} M_{\mu+n\beta}$ and denote by $\hF_{\beta,n}\colon M_{\mu+n\beta}\to \hM_\mu$ the embedding of the corresponding direct summand. Define $\hF_\mu\colon M_{\delta\mu}\to \hM_\mu$ as the row vector with entries $\hF_{\beta,n}$\footnote{The author is aware of the fact that this looks rather silly. There is a tautological identification $M_{\delta\mu}=\widehat M_\mu$ that identifies $\hF_\mu$ with the identity. However, $M_{\delta\mu}$ and $\widehat M_\mu$ will play very different roles in the following.}. For $\alpha\in\Pi$, $m>0$  define an $\SA$-linear map $\widehat E_{\alpha,m}\colon \hM_\mu\to M_{\mu+m\alpha}$ by additive extension of the following formulas. For $\beta\in\Pi$, $n>0$ and $v\in M_{\mu+n\beta}$ set
$$
\widehat E_{\alpha,m}\widehat F_{\beta,n}(v):=
\begin{cases}
F_{\beta,n}E_{\alpha,m}(v),&\text{ if $\alpha\ne\beta$},\\
\sum_{0\le r\le \min(m,n)} \qchoose{\lgl\mu,\alpha^\vee\rgl+n+m}{ r}_{d_\alpha} F_{\alpha,n-r}E_{\alpha,m-r}(v), &\text{ if $\alpha=\beta$}.
\end{cases}
$$
(Note that, in contrast to the definition in (X2), we have $v\in M_{\mu+n\beta}$, hence the $+$ in front of $n$.)
Let $\hE_\mu\colon \hM_\mu\to M_{\delta\mu}$ be the column vector with entries $\hE_{\alpha,m}$. 

Now define $\phi\colon \hM_\mu\to M_\mu$ as the row vector with entries $F_{\alpha,n}$.  Obviously, the diagram

\centerline{
\xymatrix{
&M_{\delta\mu}\ar[dl]_{\hF_\mu}\ar[dr]^{F_\mu}&\\
\hM_\mu\ar[rr]^\phi&& M_\mu
}
}
\noindent
commutes. As the $\hE_{\alpha,m}$- and $\hF_{\beta,n}$-maps satisfy the same commutation relations as the $E_{\alpha,m}$- and $F_{\beta,n}$-maps by (X2), and as $\hF_{\mu}$ is surjective, also the diagram

\centerline{
\xymatrix{
&M_{\delta\mu}&\\
\hM_\mu\ar[rr]^\phi\ar[ur]^{\hE_\mu}&& M_\mu\ar[ul]_{E_\mu}
}
}
\noindent
commutes. 
 As $\hF_\mu$ is surjective, we have $\im\phi=\im F_\mu$. As $E_\mu$ is injective when restricted to $\im F_\mu$, we deduce that $\ker\phi=\ker \hE_\mu$, hence $\phi$ induces an isomorphism $\hM_\mu/\ker \hE_\mu\cong\im F_\mu$.

Now let  $\hf_\mu\colon \hM_\mu\to N_\mu$ be the row vector with entries $F^N_{n,\beta}\circ f^\prime_{\mu+n\beta}\colon M_{\mu+n\beta}\to N_{\mu+n\beta}\to N_\mu$. Then   the diagram
 
\centerline{
\xymatrix{
M_{\delta\mu}\ar[d]_{\hF_\mu}\ar[r]^{f^\prime_{\delta\mu}}&N_{\delta\mu}\ar[d]^{F^N_\mu}\\
\hM_\mu\ar[r]^{\hf_\mu}& N_\mu
}
}
\noindent commutes. By the same arguments as above, also the diagram

\centerline{ 
\xymatrix{
M_{\delta\mu}\ar[r]^{f^\prime_{\delta\mu}}&N_{\delta\mu}\\
\hM_\mu\ar[u]^{\hE_\mu}\ar[r]^{\hf_\mu}& N_\mu\ar[u]_{E^N_\mu}
}
}
\noindent commutes.
As $\hF_\mu$ is surjective, the image of $\hf_\mu$ is contained in $\im F_\mu^N\subset N_\mu$. As $E_\mu^N$ is injective when restricted to $\im F_\mu^N$, we deduce that $\hf_\mu$ factors over the kernel of $\hE_\mu$. But, as we have seen above, this is the kernel of $\phi$. We hence obtain an induced homomorphism $\tf_\mu\colon \im F^M_\mu\cong\hM_\mu/\ker\phi\to N_\mu$ such that the diagrams

\centerline{
\xymatrix{
M_{\delta\mu}\ar[d]_{F_\mu^M}\ar[r]^{f^\prime_{\delta\mu}}&N_{\delta\mu}\ar[d]^{F^N_\mu}\\
\im F_\mu^M\ar[r]^{\tf_\mu}& N_\mu
}
\quad\quad
\xymatrix{
M_{\delta\mu}\ar[r]^{f^\prime_{\delta\mu}}&N_{\delta\mu}\\
\im F_\mu^M\ar[u]^{E^M_\mu}\ar[r]^{\tf_\mu}& N_\mu\ar[u]_{E^N_\mu}
}
}
\noindent commute. This shows the existence part of (1). The uniqueness is clear, as $F_\mu^M \colon M_{\delta\mu}\to \im F_\mu^M$ is surjective. 

Now we show part (2). Assume that property (a) holds, i.e.~there exists a homomorphism $f\colon M\to N$ that restricts to $f^\prime$. Then the diagram
 
\centerline{ 
\xymatrix{
M_{\delta\mu}\ar[rr]^{f^\prime_{\delta\mu}=f_{\delta\mu}}&&N_{\delta\mu}\\
M_\mu\ar[u]^{E^M_\mu}\ar[rr]^{f_\mu}&& N_\mu\ar[u]_{E^N_\mu}
}
}
\noindent commutes and hence $f^\prime_{\delta\mu}$ maps $M_{\lgl\mu\rgl}=E^M_\mu(M_\mu)$ into $N_{\lgl\mu\rgl}=E^N_\mu(N_\mu)$, so property (b) holds.

Now assume property (b) holds. We now need to construct an $\SA$-linear map $f_\mu\colon M_\mu\to N_\mu$ such that the diagrams

\begin{equation}\label{eq1}
\begin{gathered}
\xymatrix{
M_{\delta\mu}\ar[d]_{F^M_\mu}\ar[r]^{f^\prime_{\delta\mu}}&N_{\delta\mu}\ar[d]^{F_\mu^N}\\
M_\mu\ar[r]^{f_\mu}& N_\mu
}
\quad\quad
\xymatrix{
M_{\delta\mu}\ar[r]^{f^\prime_{\delta\mu}}&N_{\delta\mu}\\
M_\mu\ar[u]^{E_\mu^M}\ar[r]^{f_\mu}& N_\mu\ar[u]_{E_\mu^N}
}
\end{gathered}
\end{equation}
commute. By part (1), there exists a homomorphism $\tf_\mu\colon \im F_\mu^M\to N_\mu$ such that the diagrams

$$
\begin{gathered}
\xymatrix{
M_{\delta\mu}\ar[d]_{F_\mu^M}\ar[r]^{f^\prime_{\delta\mu}}&N_{\delta\mu}\ar[d]^{F_\mu^N}\\
\im F_\mu^M\ar[r]^{\tf_\mu}& N_\mu
}
\quad\quad
\xymatrix{
M_{\delta\mu}\ar[r]^{f^\prime_{\delta\mu}}&N_{\delta\mu}\\
\im F_\mu^M\ar[u]^{E_\mu^M}\ar[r]^{\tf_\mu}& N_\mu\ar[u]_{E_\mu^N}
}
\end{gathered}
$$
commute. By assumption, the quotient $M_\mu/\im F^M_\mu$ is a free $\SA$-module. We can hence fix a decomposition $M_\mu=\im F^M_\mu\oplus D$ with a free $\SA$-module $D$. We now construct a homomorphism $\hf_\mu\colon D\to N_\mu$ in such a way that  $f_\mu:=(\tf_\mu,\hf_\mu)$ serves our purpose. Note that no matter how we define $\hf_\mu$, we will always have $f_\mu\circ F_\mu^M=F_\mu^N\circ f_{\delta_\mu}^\prime$ (cf.~the left diagram in (\ref{eq1})). So the only property that $\hf_\mu$ has to satisfy is that the diagram

\centerline{
\xymatrix{
M_{\delta\mu}\ar[r]^{f^\prime_{\delta\mu}}&N_{\delta\mu}\\
D\ar[u]^{E_\mu^M|_D}\ar[r]^{\hf_\mu}& N_\mu\ar[u]_{E_\mu^N}
}
}
\noindent commutes. Since we assume that $f^\prime_{\delta\mu}(E_\mu^M(M_{\mu}))$ is contained in the image of $E^N_\mu\colon N_\mu\to N_{\delta\mu}$, this also holds for $f^\prime_{\delta\mu}(E^M_\mu(D))$.  As $D$ is free, it is projective as an $\SA$-module. So $\hf_\mu$ indeed exists.
\end{proof}
\begin{remark} In part (2) of the lemma above, the extension $f$ of $f^\prime$ is in general not unique. In the notation of the proof of part (2), the $\SA$-linear homomorphism $\hf_\mu$ is in general not unique, nor is the decomposition $M_\mu=\im F_\mu^M\oplus D$.
\end{remark}

\subsection{Soft,  saturated and standard objects}
 Let $I$ be a closed subset of $X$, let  $M$ be an object in $\CX_I$. Let $\mu\in I$. We denote by $\SA^{\irr}$ the set of irreducible elements in $\SA$. For any non-empty subset $\Gamma$ of $\SA^{\irr}$ we define
$$
M_{\{\mu\},\Gamma}:=\{m\in M_{\delta\mu}\mid \xi m \in M_{\{\mu\}}\text{ for a product $\xi$ of elements in $\Gamma$}\}.
$$
It is convenient to set $M_{\{\mu\},\emptyset}=M_{\{\mu\}}$. 
For $\Gamma\subset\Gamma^\prime\subset\SA^{\irr}$ we then  have 
$$
M_{\{\mu\}}\subset M_{\{\mu\},\Gamma}\subset M_{\{\mu\},\Gamma^\prime}\subset M_{\delta\mu},
$$
and our earlier definition reads $M_{\{\mu\},\max}:=M_{\{\mu\},\SA^{\irr}}$.

 \begin{definition} $M$ is called 
\begin{enumerate}
\item 
 {\em $\Gamma$-soft}, if  for all $\mu\in I$ we have $M_{\lgl\mu\rgl}\subset M_{\{\mu\},\Gamma}$.
 \item 
 {\em $\Gamma$-saturated}, if  for all $\mu\in I$ we have $M_{\{\mu\},\Gamma}\subset M_{\lgl\mu\rgl}$.
 \item {\em $\Gamma$-standard},  if  for all $\mu\in I$ we have $M_{\lgl\mu\rgl}=M_{\{\mu\},\Gamma}$.
 \end{enumerate}
\end{definition}
In particular, each object in $\CX_I$ is $\SA^{\irr}$-soft and $\emptyset$-saturated. The following result  might explain the terminology. 
\begin{lemma}\label{lemma-softextmor} Suppose that  $I^\prime\subset X$ is closed and that $\mu\not\in I^\prime$ is such that $I:=I^\prime\cup\{\mu\}$ is closed in $X$ as well. 
 Let $M$ and $N$ be objects in $\CX_I$ and suppose that there exists a  subset $\Gamma$ of $\SA^{\irr}$ such that $M_{\lgl\mu\rgl}\subset M_{\{\mu\},\Gamma}$ and $N_{\{\mu\},\Gamma}\subset N_{\lgl\mu\rgl}$.  Then the functorial map
$$
\Hom_{\CX_I}(M,N)\to \Hom_{\CX_{I^\prime}}(M_{I^\prime},N_{I^\prime})
$$
is surjective. 
\end{lemma}

\begin{proof} Let $f^\prime\colon M_{I^\prime}\to N_{I^\prime}$ be a morphism in $\CX_{I^\prime}$.  By Proposition \ref{prop-mainext}, there exists a (unique) $\tf_\mu\colon \im F^M_\mu\to N_\mu$ such that the diagrams 

\centerline{
\xymatrix{
M_{\delta\mu}\ar[d]_{F_\mu^M}\ar[r]^{f^\prime_{\delta\mu}}&N_{\delta\mu}\ar[d]^{F_\mu^N}\\
\im F_\mu^M\ar[r]^{\tf_\mu}& N_\mu
}
\quad\quad
\xymatrix{
M_{\delta\mu}\ar[r]^{f^\prime_{\delta\mu}}&N_{\delta\mu}\\
\im F_\mu^M\ar[u]^{E_\mu^M}\ar[r]^{\tf_\mu}& N_\mu\ar[u]_{E_\mu^N}
}
}
\noindent commute. This implies that  $f^\prime_{\delta\mu}$ maps $M_{\{\mu\}}$  into $N_{\{\mu\}}$ and hence $M_{\{\mu\},\Gamma}$  into $N_{\{\mu\},\Gamma}$. Our assumptions now imply that $f^\prime_{\delta\mu}$ maps  $M_{\lgl\mu\rgl}$ into $N_{\lgl\mu\rgl}$, so the condition (2b)  in Proposition \ref{prop-mainext} is satisfied. Hence there exists an extension $f\colon M\to N$ of $f^\prime$. 
\end{proof}

\section{Extending objects}\label{sec-ExtObj}
Again we retain all notations. In the last section we studied assumptions that ensure that morphisms in $\CX$ can be extended. In this section we want to extend objects: given a subset   $I^\prime$ of  $I$ and an object $M^\prime$  in $\CX_{I^\prime}$ we want to find an object $M$ such that $M_{I^\prime}$ is isomorphic to $M^\prime$. Such extensions always exist, and we can even order them according to the torsion type of the quotients $M_{\lgl\mu\rgl}/M_{\{\mu\}}$ for all $\mu\in I\setminus I^\prime$. For any subset $\Gamma$ of $\SA^{\irr}$ we are going to construct a minimal extension of $M^\prime$ that ``has no $\Gamma$-torsion'' at all weights in $I\setminus I^\prime$, i.e.~that satisfies $M_{\{\mu\},\Gamma}\subset M_{\lgl\mu\rgl}$ for all $\mu\in I\setminus I^\prime$.

\subsection{The $\emptyset$-extension}\label{subsec-minext}  Let  $I^\prime$ be a closed subset of $X$ and assume that $\mu\in X\setminus I^\prime$ is such that $I:=I^\prime\cup\{\mu\}$ is closed in $X$ again. We start with extending the object $M^\prime$ in the extreme case $\Gamma=\emptyset$. 
 
\begin{proposition}\label{prop-softextobj} Let $M^\prime$ be an object in $\CX_{I^\prime}$. 
\begin{enumerate}
\item There exists an up to isomorphism unique object $M$ in $\CX_I$ with the following properties.
\begin{enumerate}
\item The object $M$ restricts to $M^\prime$, i.e.~$M_{I^\prime}\cong M^\prime$. 
\item For all objects $N$ in $\CX_I$  the functorial homomorphism
$$
\Hom_{\CX_I}(M,N)\to \Hom_{\CX_{I^{\prime}}}(M_{I^{\prime}},N_{I^{\prime}})
$$
is an isomorphism. 
\end{enumerate}
\item For the object $M$ characterized in part (1) we have $M_{\{\mu\}}=M_{\lgl\mu\rgl}$.
\end{enumerate}
\end{proposition}

\begin{proof} Note that the uniqueness statement in (1) follows directly  from properties (1a) and  (1b).  So, in order to prove (1), we   only need to show the existence of $M$.  For this we  give an explicit construction. First, we  set $M_\nu=M^\prime_\nu$, $E_{\nu,\alpha,n}^{M}=E_{\nu,\alpha,n}^{M^\prime}$ and $F_{\nu,\alpha,n}^{M}=F_{\nu,\alpha,n}^{M^\prime}$ for all $\nu\in I^\prime$, $\alpha\in\Pi$ and $n>0$ in order to make sure that (1a) is satisfied.  
Then we can already define $M_{\delta\mu}:=\bigoplus_{\alpha\in\Pi, n>0} M_{\mu+n\alpha}$. For the construction of $M_\mu$ and $E_{\mu,\alpha,n}$ and $F_{\mu,\alpha,n}$ we follow ideas that were already used  in the proof of Proposition \ref{prop-mainext}. So in a first step we  set $\hM_\mu:=\bigoplus_{\beta\in\Pi,n>0} M_{\mu+n\beta}$ and denote by $\hF_{\beta,n}\colon M_{\mu+n\beta}\to \hM_{\mu}$ the canonical injection of a direct summand. We let $\hF_\mu\colon M_{\delta\mu}\to\hM_\mu$ be the row vector with entries $\hF_{\beta,n}$.  For $\alpha\in\Pi$, $m>0$  define an $\SA$-linear map $\widehat E_{\alpha,m}\colon \hM_\mu\to M_{\mu+m\alpha}$ by additive extension of the following formulas. For $\beta\in\Pi$, $n>0$ and $v\in M_{\mu+n\beta}$ set
$$
\widehat E_{\alpha,m}\widehat F_{\beta,n}(v):=
\begin{cases}
F_{\beta,n}E_{\alpha,m}(v),&\text{ if $\alpha\ne\beta$},\\
\sum_{0\le r\le \min(m,n)} \qchoose{\lgl\mu,\alpha^\vee\rgl+n+m}{ r} F_{\alpha,n-r}E_{\alpha,m-r}(v), &\text{ if $\alpha=\beta$}.
\end{cases}
$$
We denote by $\hE_\mu\colon \hM_\mu\to M_{\delta\mu}$ the column vector with entries $\hE_{\alpha,n}$. 
Now define $M_\mu:=\hM_\mu/\ker \hE_\mu$, and denote by $E_\mu\colon M_\mu\to M_{\delta\mu}$ and $F_\mu\colon M_{\delta\mu}\to M_\mu$ the homomorphisms induced by $\hE_\mu$ and $\hF_\mu$, resp. Denote by $E_{\mu,\alpha,n}\colon M_\mu\to M_{\mu+n\alpha}$ and by $F_{\mu,\alpha,n}\colon M_{\mu+n\alpha}\to M_\mu$ the entries of the row vector $E_\mu$ and the column vector $F_\mu$, resp. We now claim that the above data  yields  an object in $\CX_I$. Clearly, property (X1) is satisfied. Also, the commutation relations between the $E$- and $F$-maps follow from the resp. relations satisfied by $M^\prime$ and the construction of  $E_\mu$ and $F_\mu$. Hence (X2) is satisfied as well. The properties (X3) are satisfied for all weights $\nu$ with $\nu\ne\mu$, as they are satisfied for $M^\prime$. For the weight $\mu$, however, we have  $\ker E_\mu=\{0\}$, hence (X3a) is satisfied, and $M_\mu=\im F_\mu$, so $M_{\lgl\mu\rgl}=M_{\{\mu\}}$, which imply (X3b) and (X3c). 

It remains to show that the object $M$ satisfies the properties (1a) and (1b). Part (1a) is clear from the construction. Part (1b) follows from $M_\mu=\im F_\mu$ and part (1) of Proposition  \ref{prop-mainext}.  Hence (1) is proven. Since $M_{\mu}=\im F_\mu$ we have $M_{\{\mu\}}=M_{\lgl\mu\rgl}$, hence (2). 
\end{proof}

\subsection{Projective covers over local rings} For the next results we have to assume that projective covers exist in the category of $\SA$-modules, so we assume that $\SA$ is a local ring. Here is a short reminder on projective covers. Let $\SR$ be a ring and $M$ an $\SR$-module. Recall that a {\em projective cover} of $M$ is a surjective homomorphism $\phi\colon P\to M$ such that $P$ is a projective $\SR$-module   and such that  any submodule $U\subset P$ with $\phi(U)=M$ satisfies $U=P$. If $\SR$ is a local ring, then projective covers exist for finitely generated $\SR$-modules. They can be constructed as follows. Denote by $\SF$ the residue field of $\SR$. For an $\SR$-module $N$ we let $\ol N=N\otimes_\SR\SF$ be the associated $\SF$-vector space. In the situation above, choose an isomorphism $\SF^{n}\cong \ol M$. This can be lifted to a homomorphism $\phi\colon \SR^{n}\to M$, and Nakayama's lemma implies that this is a projective cover.

\subsection{The $\Gamma$-extension}
We let $I^\prime\subset I=I^\prime\cup\{\mu\}$ be as in Section \ref{subsec-minext}. We let $\Gamma$ be an arbitrary subset of $\SA^{\irr}$. In contrast to the case $\Gamma=\emptyset$ we now have to assume that projective covers in the category of $\SA$-modules exist. Moreover, the statement about endomorphisms in the proposition below is  slightly weaker than in the case $\Gamma=\emptyset$.  

\begin{proposition} \label{prop-GammaExtObj} Assume that $\SA$ is  local. 
 Let  $M^\prime$ be an object in $\CX_{I^\prime}$, and let $\Gamma$ be  a subset of $\SA^{\irr}$.  Then there exists an up to isomorphism unique  object $M$ in $\CX_I$ with the following properties.
\begin{enumerate}
\item $M_{I^\prime}$ is isomorphic to $M^\prime$.
\item An endomorphism $f\colon M\to M$ in $\CX_I$ is an automorphism if and only if $f_{I^\prime}\colon M_{I^{\prime}}\to M_{I^\prime}$ is an automorphism.
\item $M_{\lgl\mu\rgl}=M_{\{\mu\},\Gamma}$. 

\end{enumerate}
\end{proposition}

\begin{proof}
First, let us prove that an object $M$ having the properties (1), (2), and (3) is unique. So suppose that $M_1$ and $M_2$ have these properties. Then, by (1), we have an isomorphism $M_{1I^\prime}\cong M_{2I^\prime}$.  From Proposition \ref{prop-mainext} we deduce that this isomorphism identifies $M_{1\{\mu\}}$ with $M_{2\{\mu\}}$ and hence $M_{1\{\mu\},\Gamma}$ with $M_{2\{\mu\},\Gamma}$, so $M_{1\lgl\mu\rgl}$ with $M_{2\lgl\mu\rgl}$ by property (3). So the condition in Lemma \ref{lemma-softextmor} is satisfied, so  the chosen isomorphism extends to a homomorphism $f\colon M_1\to M_2$. Reversing the roles of $M_2$ and $M_1$ yields a homomorphism $g\colon M_2\to M_1$ in an analogous way. Now property (2) implies that $g\circ f$ and $f\circ g$ are automorphisms. Hence $f$ and $g$ are isomorphisms.

It remains to show that an object $M$ with  properties (1), (2) and (3) exists. Denote by $\tM$ the object in $\CX_I$ that extends $M^\prime$ in the sense of Proposition \ref{prop-softextobj}. Then we can identify  $\tM_{\delta\mu}$ with $M^\prime_{\delta\mu}$. We set  $Q:=\tM_{\{\mu\},\Gamma}/\tM_{\{\mu\}}$, so this is a torsion $\SA$-module. It is finitely generated as it is also a quotient of $\tM_\mu$. Now we fix a projective cover  $\ol\gamma\colon D\to Q$ in the category of $\SA$-modules, and we denote by $\gamma\colon D\to \tM_{\{\mu\},\Gamma}$ a lift of $\ol\gamma$. We can also  consider $\gamma$ as a homomorphism from $D$ to $\tM_{\delta\mu}$. 

We define $M$ as follows. We set $M_\nu:=\tM_\nu$ for all $\nu\in I^\prime$, $E^M_{\nu,\alpha,n}:=E^{\tM}_{\nu,\alpha,n}$ and $F^M_{\nu,\alpha,n}:=F^{\tM}_{\nu,\alpha,n}$. Then we set $M_\mu:=\tM_\mu\oplus D$ and define $F^M_{\mu,\alpha,n}:=(F^{\tM}_{\mu,\alpha,n},0)^T\colon M_{\delta\mu}=\tM_{\delta\mu}\to M_{\mu}$ and $E^{M}_{\mu,\alpha,n}:=(E^{\tM}_{\mu,\alpha,n},\gamma)\colon M_{\mu}\to M_{\delta\mu}=\tM_{\delta\mu}$. 
We now show that $M=\bigoplus_{\nu\in I}M_\mu$ together with the $E$- and $F$-maps above is an object in $\CX_I$. Condition  (X1) is clearly satisfied. We now show that (X2) is also satisfied. Let $\nu\in I$, $\alpha,\beta\in\Pi$, $m,n>0$ and $v\in M_{\nu+n\beta}$. We need to show that 

\begin{align*}
E_{\alpha,m}F_{\beta,n}(v)=
\begin{cases}
F_{\beta,n}E_{\alpha,m}(v),&\text{ if $\alpha\ne\beta$}\\
\sum_{0\le r\le \min(m,n)} \qchoose{\lgl\nu,\alpha^\vee\rgl+m+n}{r}_{d_\alpha}F_{\alpha,n-r}E_{\alpha,m-r}(v), &\text{ if $\alpha=\beta$}.
\end{cases}
\end{align*}
If $\nu\ne\mu$, then this follows immediately from the fact that (X2) is satisfied for $\tM$, and in the case $\nu=\mu$ it follows as $E^M_{\mu,\alpha,m}$ coincides with $E^{\tM}_{\mu,\alpha,m}$ on the image of $F_{\mu,\beta,n}$.  
We now check the condition (X3). It  is satisfied for all $\nu\ne\mu$, as it is satisfied for $\tM$. In the case $\nu=\mu$, (X3a) follows from the corresponding condition for $\tM$ as the image of $F_\mu^M$ coincides with the image of $F_\mu^{\tM}$ and $E^M_\mu$ agrees with $E^{\tM}_\mu$ on this image. By construction $M_{\lgl\mu\rgl}=M_{\{\mu\},\Gamma}$, so  the inclusion $M_{\{\mu\}}\subset M_{\lgl\mu\rgl}$ has a torsion cokernel. Finally, we have $\im F_\mu^{M}=(\tM_\mu,0)$, so the quotient $M_\mu/\im F^M_\mu$ is isomorphic to $D$. As $D$ was chosen to be a projective $\SA$-module, it is free. So we have indeed constructed an object in $\CX_I$.

 We need to check that $M$ satisfies the properties (1a), (1b)  and (2). Clearly, $M_{I^\prime}=\tM_{I^\prime}\cong M^\prime$ so (1a) is satisfied.
We have already observed that $M_{\lgl\mu\rgl}=M_{\{\mu\},\Gamma}$, hence (2).  Now let $f\colon M\to M$ be an endomorphism and suppose that $f_{I^\prime}\colon M_{I^\prime}\to M_{I^\prime}$ is an automorphism, i.e.~$f_\nu\colon M_\nu\to M_\nu$ is an automorphism for all $\nu\ne\mu$.  Then $f_{\delta\mu}\colon M_{\delta\mu}\to M_{\delta\mu}$ is an automorphism, and hence the restriction of $f_{\mu}|_{\im F_\mu}\colon \im F_\mu\to \im F_\mu$  is an automorphism.  Applying $E_\mu$ shows that $f_{\{\mu\}}$ is an automorphism of $M_{\{\mu\}}$. Hence $f_{\delta\mu}$ induces an automorphism of $M_{\{\mu\},\Gamma}$ and we obtain an induced automorphism of the quotient $Q$ defined earlier in this proof. As $\ol\gamma\colon D\to Q$ is a projective cover, also the induced endomorphism on $D$  must be an automorphism. Hence $f_\mu$ is an automorphism, and hence so is $f$. Hence (2) also holds.
\end{proof}

\subsection{The category of $\Gamma$-standard objects}

Again we fix a subset $\Gamma$ of $\SA^{\irr}$. We can now classify a family of objects parametrized by their highest weight. 

\begin{proposition}\label{prop-catGammaStand} Suppose that $\SA$ is local in the case that  $\Gamma\ne\emptyset$. 
\begin{enumerate}
\item  For all $\lambda\in X$ there exists an up to isomorphism unique object $S_\Gamma(\lambda)$ in $\CX$ with the following properties.
\begin{enumerate}
\item $S_\Gamma(\lambda)_\lambda$ is free of rank $1$ and $S_\Gamma(\lambda)_\mu\ne\{0\}$ implies $\mu\le\lambda$. 
\item $S_\Gamma(\lambda)$ is indecomposable and $\Gamma$-standard. 
\end{enumerate}
\end{enumerate}
Moreover, the objects $S_\Gamma(\lambda)$ characterized in (1) have the following properties.
\begin{enumerate}\setcounter{enumi}{1}
\item An endomorphism $f$ of $S_\Gamma(\lambda)$ is an automorphism if and only if it restricts to an automorphism on the $\lambda$-weight space. In the case  $\Gamma=\emptyset$, we even have   $\End_\CX(S_\emptyset(\lambda))=\SA\cdot\id$. 
\item Let $S$ be a $\Gamma$-standard object in $\CX_\SA$. Then there is an index set $J$ and some elements $\lambda_i\in X$ for $i\in J$ such that $S\cong \bigoplus_{i\in J}S_\Gamma(\lambda_i)$. 
\end{enumerate}
\end{proposition}
Sometimes we will write $S_{\Gamma\SA}(\lambda)$ to incorporate the ground ring.

\begin{proof} We start with proving that there exists an object $S_\Gamma(\lambda)$ satisfying the properties (1a) and (1b) as well as (2).  We then show that (3) holds with this particular set of objects $S_\Gamma(\lambda)$, which then implies the remaining uniqueness statement in (1). So  set ${I_\lambda}:=X\setminus\{<\lambda\}=\{\mu\in X\mid \mu\not <\lambda\}$. This is a closed subset of $X$ that contains $\lambda$ as a minimal element. Define $S^\prime=\bigoplus_{\mu\in {I_\lambda}}S^\prime_\mu$ by setting $S^\prime_\lambda=\SA$ and $S^\prime_\mu=\{0\}$ for $\mu\in I_\lambda\setminus\{\lambda\}$.  For any $\mu\in I_\lambda$, $\alpha\in\Pi$, $n>0$ we have  $S^\prime_{\mu+n\alpha}=0$, and so all the maps $F_{\mu,\alpha,n}$ and $E_{\mu,\alpha,n}$ are the zero homomorphisms. Then $S^\prime$ is an object in $\CX_{I_\lambda}$. 

Note that for any $\mu\in X\setminus I_\lambda$ the set $\{\nu\in X\setminus I_\lambda\mid \mu\le\nu\}$ is  contained in the interval $[\mu,\lambda)$, hence is finite. We can hence  employ Proposition \ref{prop-softextobj} in the case $\Gamma=\emptyset$ and Proposition \ref{prop-GammaExtObj} in the other cases inductively on the partially ordered set $X\setminus I_\lambda$. We   deduce that there exists an object $S_\Gamma(\lambda)$ in $\CX$ such that $S_\Gamma(\lambda)_{I_\lambda}=S^\prime$ (this implies property (1a)) and such that for $\mu\in X$, $\mu\not\in I_\lambda$ we have $S_\Gamma(\lambda)_{\lgl\mu\rgl}=S_{\Gamma}(\lambda)_{\{\mu\},\Gamma}$. Both objects vanish for  $\mu\in I_\lambda$ by the definition of $S^\prime$. Hence $S_\Gamma(\lambda)$ is $\Gamma$-standard.  Proposition \ref{prop-softextobj}, or  Proposition \ref{prop-GammaExtObj}, resp.,  now  imply, by induction,  the statement (2) for the particular object $S_\Gamma(\lambda)$ that we just constructed. It follows that  $S_\Gamma(\lambda)$ is indecomposable. So property (1b) holds for $S_\Gamma(\lambda)$ as well. 

We are now left with proving  property (3), where we assume that the  $S_\Gamma(\lambda)$ appearing in the statement are the objects we just constructed explicitely.  So let $S$ be a $\Gamma$-standard object, and fix a maximal weight $\lambda$ of $S$ (which exists by (X1)). Then $S_\lambda$ is a free $\SA$-module by assumption (X3c).  Set $I^\prime_\lambda:=\{\mu\in X\mid \lambda\le\mu\}$ (a closed subset of $X$) The maximality of $\lambda$ implies that the restriction $S_{I^\prime_\lambda}$ in $\CX_{I^\prime_\lambda}$ is isomorphic to a direct sum $\tilde S_{I^\prime_\lambda}$, where $\tilde S$ is isomorphic to  a direct sum of  copies of $S_\Gamma(\lambda)$.  Since both $S$ and  $\tilde S$ are $\Gamma$-standard,  Lemma \ref{lemma-softextmor} implies  that the identification of the $I^\prime_\lambda$-restrictions extend, so there are morphisms $f\colon \tilde S\to S$ and $g\colon S\to \tilde S$ such that $(g\circ f)|_{I^\prime_\lambda}$ is the identity. We deduce that $g\circ f$ is an automorphism. Hence $\tilde S$ is isomorphic to a direct summand of $S$. By construction, $\lambda$ is not a weight of a direct complement. From here we can continue by induction to prove (3). 
\end{proof}
\begin{remark}\label{rem-algo} Note that the main ingredient in the existence result  above is Proposition \ref{prop-GammaExtObj}. The proof of that proposition  is constructive, i.e.~it can be read as an algorithm to construct the weight spaces of the objects $S_{\Gamma}(\lambda)$ inductively, starting with the highest weight space.
\end{remark}

Now assume for a moment that $\SA=\SK$ is a field. 

\begin{lemma}\label{lemma-fieldcase} Any object in $\CX_\SK$ is isomorphic to a direct sum of copies of the $\emptyset$-standard objects $S_\emptyset(\lambda)$ and hence $\CX_\SK$ is a semisimple category.
\end{lemma}
\begin{proof}
For any object $M$ in $\CX_\SK$ we have $M_\mu=\im F_\mu\oplus\ker E_\mu$ by Lemma \ref{lemma-X3}. It follows from Lemma \ref{lemma-soft} that $M$ is $\emptyset$-standard and the claim follows from Proposition \ref{prop-catGammaStand}.
\end{proof}

\subsection{Base change}\label{sec-basech} We now want to understand whether the conditions that define the category $\CX$ are stable under base change. 
So let $\SA\to\SB$  be a homomorphism of unital $\SZ$-algebras (that are factorial domains). Let $M$ be an object in $\CX_{\SA}$. We define $M_{\SB}=\bigoplus_{\mu\in X} M_{\SB\mu}$ by setting $M_{\SB\mu}:=M_\mu\otimes_\SA\SB$. For $\mu\in X$, $\alpha\in\Pi$ and $n>0$ we then have induced homomorphisms $E^{M_{\SB}}_{\mu,\alpha,n}=E_{\mu,\alpha,n}\otimes\id_{\SB}\colon M_{\SB\mu}\to M_{\SB\mu+n\alpha}$ and $F^{M_{\SB}}_{\mu,\alpha,n}=F_{\mu,\alpha,n}\otimes\id_{\SB}\colon M_{\SB\mu+n\alpha}\to M_{\SB\mu}$.

\begin{lemma} \label{lemma-basechange} Suppose that $\SA\to\SB$ is a flat homomorphism. Then the  object $M_\SB$ is contained in $\CX_\SB$.
\end{lemma}
\begin{proof} It is clear that the properties (X1) and (X2) are stable under arbitrary base change. Moreover,  $M_{\SB\lgl\mu\rgl}$, $M_{\SB\{\mu\}}$ and  $\im F_{\mu}^{M_\SB}$  are obtained from $M_{\lgl\mu\rgl}$, $M_{\{\mu\}}$ and $\im F_\mu^M$ by base change. Hence (X3b) and (X3c) also hold for $M_\SB$. By the flatness condition, the homomorphism $E_\mu|_{\im F_\mu}$ remains injective after base change. Hence property  (X3a) also holds.  \end{proof}

\begin{lemma}\label{lemma-basechangeS} Suppose that $\SA\to\SB$ is a flat homomorphism of local $\SZ$-algebras. Let $\Gamma$ be a subset of $\SA^{\irr}$ and let $\Gamma^\prime$ be the image of $\Gamma$ in $\SB$. Suppose that $\Gamma^\prime$ does not contain $0$.  Let $\lambda\in X$. 
\begin{enumerate}
\item The object $S_{\Gamma,\SA}(\lambda)_\SB$  is $\Gamma^\prime$-standard (but in general  decomposable). 
\item For $\Gamma=\emptyset$ we have $S_{\emptyset,\SA}(\lambda)_\SB\cong S_{\emptyset,\SB}(\lambda)$. 
\end{enumerate}
\end{lemma}
\begin{proof}  As we have a flat extension, the object $S_{\Gamma,\SA}(\lambda)_\SB$ is contained in $\CX_\SB$ by Lemma \ref{lemma-basechange}. The condition $M_{\{\mu\},\Gamma}=M_{\lgl\mu\rgl}$ is stable  under base change, so $S_{\Gamma,\SA}(\lambda)_{\SB}$ is $\Gamma^\prime$-standard in $\CX_\SB$. Proposition  \ref{prop-catGammaStand} implies that it is isomorphic to a direct sum of various $S_{\Gamma^\prime,\SB}(\mu)$'s. Now suppose that  $\Gamma=\emptyset$. As $S_{\emptyset,\SA}(\lambda)$ is generated by its $\lambda$-weight space\footnote{i.e., the smallest $X$-graded subspace of $S_{\emptyset,\SA}(\lambda)$ that contains $S_{\emptyset,\SA}(\lambda)_\lambda$ and is stable under all $E$- and $F$-maps,  is $S_{\emptyset,\SA}(\lambda)$}, which is of rank $1$, and $\lambda$ is maximal among the weights. Hence the same holds for  $S_{\emptyset,\SA}(\lambda)_{\SB}$. This implies that it is isomorphic to $S_{\emptyset,\SB}(\lambda)$.  
 \end{proof}

\section{Representations of quantum groups}\label{sec-qg}

In this section we show that the category $\CX$ has an interpretation in terms of quantum group representations. The main reasons for this are the uniqueness of the $\emptyset$-standard objects $S_\emptyset(\lambda)$ and the semisimplicity statement in Lemma \ref{lemma-fieldcase}. 

\subsection{Quantum groups over $\SZ$-algebras}
We denote by $U_\SZ$ the quantum group over $\SZ=\DZ[v,v^{-1}]$ (with divided powers) associated with the Cartan matrix $(\lgl\alpha,\beta^\vee\rgl)_{\alpha,\beta\in\Pi}$ of $R$. Its definition by generators and relations can be found  \cite[Sections 1.1-1.3]{L2}. We denote by $e^{[n]}_\alpha,  f^{[n]}_\alpha, k_\alpha, k_\alpha^{-1}$ for $\alpha\in\Pi$ and $n>0$ the standard generators of $U_\SZ$. 

For  $\alpha\in R$,  $n>0$ also the element
$$ 
\qchoose{k_\alpha}{n}_{\alpha}:=\prod_{s=1}^n \frac{k_\alpha v_\alpha^{-s+1}-k_{\alpha}^{-1}v_\alpha^{s-1}}{v_\alpha^{s}-v_\alpha^{-s}} 
$$
is contained in $U_\SZ$ (where $v_\alpha:=v^{d_\alpha}$). We let $U_\SZ^+$, $U_\SZ^-$ and $U_\SZ^0$ be the unital subalgebras of $U_\SZ$ that are generated by the sets $\{e_\alpha^{[n]}\}$, $\{f_\alpha^{[n]}\}$ and $\{k_\alpha, k_\alpha^{-1},\qchoose{k_\alpha}{n}_{\alpha}\}$, resp. A remarkable fact, proven by Lusztig, is that each of these subalgebras is free over $\SZ$ and admits a PBW-type basis, and that the multiplication map $U^-_\SZ\otimes_\SZ U_\SZ^0\otimes_\SZ U_\SZ^+\to U_\SZ$ is an isomorphism of $\SZ$-modules (Theorem 6.7 in \cite{L2}). 

For a $\SZ$-algebra $\SA$ as above we set $U_\SA:=U_\SZ\otimes_\SZ\SA$ and $U_\SA^\ast:= U_\SZ^\ast\otimes_\SZ\SA$ for $\ast=-,0,+$.
In this article we consider $U_\SA$ only as an associative, unital algebra and forget about the Hopf algebra structure.

\subsection{The category $\CO_\SA$}
By \cite[Lemma 1.1]{APW} every $\mu\in X$ yields a character
\begin{align*}
\chi_\mu\colon U^0_\SZ&\to\SZ\\
k_\alpha^{\pm1}&\mapsto v_\alpha^{\pm \lgl\mu,\alpha^\vee\rgl}\\
\qchoose{k_\alpha}{r}_\alpha&\mapsto \qchoose{\lgl\mu,\alpha^\vee\rgl}{r}_{d_\alpha}\text{ ($\alpha\in\Pi$, $r\ge0$)}.
\end{align*}
We can extend this character to a character $\chi_\mu\colon U^0_\SA\to\SA$. 
A $U_\SA$-module $M$ is called a {\em weight module} if $M=\bigoplus_{\mu\in X}M_\mu$, where
$$
M_\mu:=\{m\in M\mid H.m=\chi_\mu(H)m\text{ for all $H\in U_\SA^0$}\}.
$$
Hence all the weight modules that we consider in this article  are ``of type 1'' (cf.~\cite[Section 5.1]{JanQG}).

\begin{definition}\label{def-Q} Let $\CO_\SA$ be the full subcategory of the category of $U_\SA$-modules that contains all objects $M$ with the following properties.
\begin{enumerate}
\item $M$ is a weight module and its set of weights is quasi-bounded from above. 
\item For each $\mu\in X$, the weight space $M_\mu$ is a finitely generated torsion free $\SA$-module. 
\end{enumerate}
\end{definition}

\begin{remark} Note that the definition above yields an $\SA$-linear category, which in general is not abelian (due to the torsion freeness assumption).  If $\SA=\SK$ is a field, then torsion freeness is always satisfied, and we obtain an abelian category.
\end{remark}

Now we establish a first, rather easy, link to the objects that we considered in the earlier chapters.
Let us denote by $\tilde\CX_\SA$ the category whose objects are  $X$-graded $\SA$-modules $M$ endowed with operators $E_{\alpha,n}$ and $F_{\alpha,n}$ as in Section \ref{sec-Xgradop} that satisfy conditions (X1) and (X2) (but not necessarily (X3)), and with morphisms being the $X$-graded $\SA$-linear homomorphisms that commute with the $E$- and $F$-maps.  

Let $M$ be an object in $\CO_\SA$. Let us denote  by $E_{\mu,\alpha,n}\colon M_\mu\to M_{\mu+n\alpha}$ and $F_{\mu,\alpha,n}\colon M_{\mu+n\alpha}\to M_\mu$ the homomorphisms given by the actions of $e_\alpha^{[n]}$ and $f_{\alpha}^{[n]}$, resp. By forgetting structure, we now consider $M$ only as an $X$-graded space endowed with these operators. 
\begin{lemma} \label{lemma-funS} The above yields a  fully faithful functor
$$
\mathsf{S}\colon \CO_\SA\to\tilde \CX_\SA.
$$
\end{lemma}
\begin{proof} It is clear that the above construction is functorial.  Let $M$ be an object in $\CO_\SA$. We need to check that the graded space with operators that we obtain from $M$ satisfies the conditions (X1) and (X2). Condition (X1) is part of the definition of $\CO_\SA$. Now we check condition (X2). Set
$$
\qchoose{k_\alpha;c}{r}_\alpha =\prod_{s=1}^r\frac{k_\alpha v_\alpha^{c-s+1}-k_{\alpha}^{-1}v_\alpha^{-c+s-1}}{v_\alpha^{s}-v_\alpha^{-s}}.
$$
This element is contained in $U^0_\SA$ and acts as multiplication with 
\begin{align*}
\prod_{s=1}^r\frac{v_\alpha^{\lgl\nu,\alpha^\vee\rgl+c-s+1}-v_\alpha^{-\lgl\nu,\alpha^\vee\rgl-c+s-1}}{v_\alpha^{s}-v_\alpha^{-s}}.
\end{align*}
on each vector of weight $\nu$. By \cite[Section 6.5]{L2} the following relations holds in $U_\SZ$ for all $\alpha,\beta\in\Pi$, $m,n>0$: 
 $$
 e_{\alpha}^{[m]}f_\beta^{[n]}=\sum_{r=0}^{\min(m,n)} f_\alpha^{[n-r]}\qchoose{k_\alpha;2r-m-n}{r}_\alpha e_\alpha^{[m-r]}.
 $$
For  $v\in M_\mu$ we hence obtain
\begin{align*}
 e_{\alpha}^{[m]}f_\beta^{[n]}(v)&=\sum_{r=0}^{\min(m,n)} f_\alpha^{[n-r]}\prod_{s=1}^r\frac{v_\alpha^{\zeta-s+1}-v_\alpha^{-\zeta+s-1}}{v_\alpha^{s}-v_\alpha^{-s}}e_\alpha^{[m-r]}(v),
 \end{align*}
 where $\zeta=\lgl\mu+(m-r)\alpha,\alpha^\vee\rgl+2r-m-n=\lgl\mu,\alpha^\vee\rgl+m-n$. In order to prove that condition (X2) holds, it remains to show that 
 $$
 \qchoose{\zeta}{r}_{d_\alpha}=\prod_{s=1}^r\frac{v_\alpha^{\zeta-s+1}-v_\alpha^{-\zeta+s-1}}{v_\alpha^{s}-v_\alpha^{-s}},
 $$
 which is (almost) immediate from the definition. Hence $\mathsf{S}$ is indeed a functor from $\CO_\SA$ to $\CX_\SA$.
 
Now $U_\SA$ is generated by the elements $e_\alpha^{[n]}$, $f_\alpha^{[n]}$ for $\alpha\in\Pi$ and $n>0$ as an algebra over $U_\SA^0$ by the PBW-theorem.  As the actions of the $e$- and $f$-elements are encoded by the $E$- and $F$-homomorphisms, and as the action of $U_\SA^{0}$ is  encoded by the $X$-gradation, the functor $\mathsf{S}$ is fully faithful. 
\end{proof}
Note that we can consider $\CX_\SA$ as a full subcategory of $\tilde\CX_\SA$. In the next section we construct a functor $\mathsf{R}\colon\CX_\SA\to\CO_\SA$ that is right inverse to $\mathsf{S}$.

\subsection{A functor from $\CX_\SA$ to  $\CO_\SA$} First we suppose that $\SA=\SK$ is a field. 
For each $\lambda$, the character $\chi_\lambda$ of $U^0_\SK$ can be uniquely extended to a character of  $U_\SK^0U_\SK^+$ such that $\chi_\lambda(e_\alpha^{[n]})=0$ for all $\alpha\in\Pi$, $n>0$. So we obtain an $U_\SK^0U_\SK^+$-module $\SK_\lambda$ of dimension $1$. We denote by $\Delta_\SK(\lambda):=U_\SK\otimes_{U_\SK^0U_\SK^+}\SK_\lambda$ the induced $U_\SK$-module (this is the {\em Verma module} with highest weight $\lambda$). It  has a unique irreducible quotient that  we denote by   $L_\SK(\lambda)$. Both are objects in $\CO_\SK$. For more information on these objects, see \cite{A,FracPerO}.

\begin{proposition} \label{prop-irreps} Suppose that $\SA=\SK$ is a field. For all $\lambda\in X$ the object  $\mathsf{S}(L_\SK(\lambda))$ is contained in the subcategory $\CX_\SK$ of $\tilde\CX_\SK$ and it is isomorphic  to $S_{\emptyset,\SK}(\lambda)$. 
\end{proposition}

\begin{proof} In view of Lemma \ref{lemma-funS} we need to check property (X3) in order to prove the first statement. Set $M=\mathsf{S}(L_\SK(\lambda))$. As $L_\SK(\lambda)$ is an irreducible $U_\SK$-module of highest weight $\lambda$,  we have $\im F_\mu=M_\mu$ for all $\mu\ne \lambda$, and $\im F_\lambda=\{0\}$. On the other hand, there are no non-trivial primitive vectors in $L_\SK(\lambda)$ of weight $\mu$ if $\mu\ne\lambda$. (A primitive vector is a vector annihilated by all $e_{\alpha}^{[n]}$.) So we have  $\ker E_\mu=\{0\}$ for $\mu\ne\lambda$ and $\ker E_\lambda=M_\lambda$. In any case we have $M_\mu=\ker E_\mu\oplus \im F_\mu$, which, by Lemma \ref{lemma-X3}, is equivalent to the set of conditions (X3). Hence $M$ is an object in $\CX_\SK$. Lemma \ref{lemma-fieldcase} now yields that $M$ is isomorphic to a direct sum of various $S_{\emptyset,\SK}(\mu)$'s. As $\mathsf{S}$ is faithful, $M$ is indecomposable, and a comparison of weights shows $M\cong S_{\emptyset,\SK}(\lambda)$. 
\end{proof}

Here is our ``realization theorem''. 
\begin{theorem}\label{thm-XU} Let $M$ be an object in $\CX_\SA$. Then there exists a unique structure of a  $U_\SA$-module on $M$ such that the following holds.
\begin{itemize}
\item The $X$-gradation $M=\bigoplus_{\mu\in X}M_\mu$ is the weight decomposition.
\item  For all $\mu\in X$, $\alpha\in\Pi$, $n>0$, the homomorphisms $E_{\mu,\alpha,n}$ and $F_{\mu,\alpha,n}$ are the action maps of $e_{\alpha}^{[n]}$ on $M_\mu$ and $f_{\alpha}^{[n]}$ on $M_{\mu+n\alpha}$, resp.
\end{itemize}
 From this we obtain a fully faithful functor $\mathsf{R}\colon\CX_\SA\to \CO_\SA$, and we have $\mathsf{S}\circ\mathsf{R}\cong\id_{\CX_{\SA}}$. 
\end{theorem}

\begin{proof} Note that the uniqueness statement in the claim above follows immediately from the fact that $U_\SA$ is generated as an algebra by the elements $e_\alpha^{[n]}$, $f_\alpha^{[n]}$ and $k_\alpha$, $k_\alpha^{-1}$ for $\alpha\in\Pi$ and $n>0$. We now prove the existence of a $U_\SA$-module structure on $M$ with the alleged properties. 

First suppose that $\SA=\SK$ is a field. Then every object in $\CX_\SK$ is isomorphic to a direct sum of various $S_{\emptyset,\SK}(\lambda)$'s by Lemma \ref{lemma-fieldcase}. By inverting the statement of Proposition \ref{prop-irreps} we see that any $S_{\emptyset,\SK}(\lambda)$ carries the structure of an  $U_\SK$-module of the required kind  (making it isomorphic to $L_\SK(\lambda)$). So the result holds in the case that $\SA$ is a field.

Now let $\SA$ be arbitrary. We denote by $\SK$ its quotient field. For any object $M$ in $\CX_\SA$, $M_\SK=M\otimes_\SA\SK$ is an object in $\CX_\SK$ by Lemma \ref{lemma-basechange}. As $M$ is a torsion free $\SA$-module by Lemma \ref{lemma-X3}, we can view $M$ as an $\SA$-submodule in $M_\SK$. Now by the above, we can view $M_\SK$ as an object in $\CO_\SK$. As $M$ is stable under the maps $E_{\alpha,n}$ and $F_{\alpha,n}$, it is stable under the action of $e_{\alpha}^{[n]}$ and $f_\alpha^{[n]}$. Moreover, it is clearly stable under the action of $k_\alpha$ and $k_\alpha^{-1}$. Hence it is stable under the action of  $U_\SA\subset U_\SK$. So there is indeed a natural $U_\SA$-module structure on $M$, and one immediately checks that this makes it into an object of category $\CO_\SA$.

Clearly the above $U_\SA$-structure depends functorially on $M$, so  we indeed obtain a functor $\mathsf{R}$ from $\CX_\SA$ to $\CO_\SA$. It is obviously fully faithful and clearly  $\mathsf{S}\circ\mathsf{R}$ is isomorphic to the identity on $\CX_\SA$. 
\end{proof}

\section{Objects in $\CO$ that  admit a  Weyl filtration}\label{sec-Weyl}
 The main goal of this section is to show that the functors $\mathsf{S}$ and $\mathsf{R}$ induce mutually inverse equivalences between the category $\CX^f$ of objects in $\CX$ that are free of finite rank over $\SA$, and the category of objects in $\CO$ that admit a (finite) Weyl filtration. In order to be able to quote some representation theoretic results on  Weyl modules, we need to assume that the quotient field of our ground ring $\SA$ is {\em generic}. For example, this implies that the Weyl modules are free of finite rank over $\SA$ and that their characters are given by Weyl's character formula. Moreover, the finite dimensional representation theory of $U_\SK$ is semisimple in this case. We do not need $\SA$ to be local, though. 

\subsection{Generic algebras}\label{subsec-groundrings}
Let $\SA$ be a unital $\SZ$-algebra that is a factorial domain.  

\begin{definition} \label{def-generic} We say that $\SA$ is {\em generic} if for all  $n\ne 0$ and all $d>0$ the image of the quantum integer $[n]_{d}$ in $\SA$ is non-zero (i.e.~invertible in the quotient field $\SK$).
\end{definition}
We denote by $q\in\SA$ the image of $v$ under the structural homomorphism $\SZ\to\SA$, $f\mapsto f\cdot 1_\SA$. Then $q$ is invertible in $\SA$, and a $\SZ$-algebra is nothing but an algebra with a specified invertible element $q$.  
Note that if $\SA$ is not generic, then $q$ is a root of unity in $\SA$, i.e.~it divides $q^{l}-1$ for some $l>1$, as 
\begin{align*}
[n]_d&=\frac{v^{dn}-v^{-dn}}{v^d-v^{-d}}
=v^{-dn+d}\frac{v^{2dn}-1}{v^{2d}-1}. 
\end{align*}
The converse is not true. For example, a field of characteristic $0$ with $q=1$ is generic. However,  a field of positive characteristic and $q=1$ is not generic, but the ring of $p$-adic integers $\DZ_p$ with $q=1$ is generic. If $\zeta\in \SK$ is a root of unity $\ne 1$, then $\SA=\SK$ with $q=\zeta$ is not generic.

Let $p$ be a prime number and denote by $\SZ_\fp$ the localization of $\SZ$ at the prime ideal 
\begin{align*}
\fp&:=\{g\in\SZ\mid \text{ $g(1)$ is divisible by $p$}\}\\
&=\ker\left( \SZ\xrightarrow{v\mapsto 1} \DF_p\right),
\end{align*}
i.e.~$\SZ_\fp=\{\frac{f}{g}\in\Quot\SZ\mid \text{ $g(1)$ is not divisible by $p$}\}$. So $\SZ_\fp$ is a generic local $\SZ$-algebra, and its residue field is $\DF_p$. Similarly, denote by $\sigma_l\in\DQ[v]$ the $l$-th cyclotomic polynomial for $l\ge 1$, and let $\DQ[v]_{(\sigma_l)}$ be the localization at the prime ideal generated by $\sigma_l$. We obtain a generic and local $\SZ$-algebra with residue field $\DQ[\zeta_l]$, the $l$-th cyclotomic field.

The last two examples show that the assumption that $\SA$ is generic is not a huge restriction even if one is interested in the  theory of tilting modules for  algebraic groups in positive characteristics or for  quantum groups at roots of unity. These objects admit deformations to the generic and local algebras $\DZ_p$ and $\DQ[v]_{(\sigma_l)}$, and  one obtains information in the non-generic cases by specializing results from generic cases.

\subsection{Modules admitting a Weyl filtration}
 Let us now assume that $\SA$   is generic. We denote by  $\SK$ its quotient field. Recall that for any $\lambda\in X$ we denote by $L_\SK(\lambda)$ the irreducible object in $\CO_\SK$  with highest weight $\lambda$. Then $L_\SK(\lambda)$ is finite dimensional if and only if $\lambda$ is dominant, i.e.~is contained in the set $X^+=\{\lambda\in X\mid \lgl\lambda,\alpha^\vee\rgl\ge 0\text{ for all $\alpha\in\Pi$}\}$, cf.~\cite[Theorem 2.3]{A}. If $\lambda$ is dominant, then we define $W_\SA(\lambda)$ as the $X$-graded $U_\SA$-submodule in $L_\SK(\lambda)$ generated by a non-zero element in $L_\SK(\lambda)_\lambda$. This is an object in $\CO_\SA$ and it does not depend, up to  isomorphism, on the choice of the element.  It is called the {\em Weyl module} with highest weight $\lambda$. The following result can be found in  \cite[Proposition 1.22]{APW}.

\begin{proposition}\label{prop-Weylchar} Assume that $\SA$ is generic. Let $\lambda\in X$ be dominant. Then $W_\SA(\lambda)$ is a free $\SA$-module of finite rank and its character is given by Weyl's character formula.
\end{proposition}

Let $M$ be an object in $\CO_\SA$. 
\begin{definition} We say that $M$ {\em admits a Weyl filtration} if there is a finite filtration $0=M_0\subset M_1\subset \dots\subset M_n=M$ and $\lambda_1,\dots,\lambda_
n\in X^+$ such that for each $i=1,\dots,n$, the subquotient $M_i/M_{i-1}$ is isomorphic to $W_\SA(\lambda_i)$. 

\end{definition}
We denote by $\CO_\SA^W$ the full subcategory of $\CO_\SA$ that contains all objects that admit a Weyl filtration. Note that if $\SA$ is generic and if $\SK$ is its quotient field, then  $\CO^W_\SK$ is a semisimple category (cf. Theorem 5.15 and Section 6.26 in \cite{JanQG}).

\subsection{A criterium for Weyl filtrations}
 Let $M$ be an object in $\CO_\SA$.

\begin{lemma}\label{lemma-freehws} Suppose that $\SA$ is generic.  Suppose that  $M$ is finitely generated as an $\SA$-module and  that there exists a dominant element $\lambda\in X$ such that the following holds.
\begin{enumerate}
\item The weight space $M_\lambda$ generates  $M$ as a $U^-_\SA$-module.
\item The weight space $M_\lambda$ is a free $\SA$-module of finite rank $r$. 
\end{enumerate}
Then $M$ is isomorphic to a direct sum of $r$ copies of $W_\SA(\lambda)$. 
\end{lemma}

\begin{proof} For all $\mu\in X$ we denote by $\Delta_\SA(\lambda):=U_\SA\otimes_{U_\SA^{\ge 0}}\SA_\lambda$ the {\em Verma module} in $\CO_\SA$ with highest weight $\lambda$. Here, $U_\SA^{\ge 0}=U^0_\SA U^+_\SA$, and $\SA_\lambda$ is the free $\SA$-module of rank $1$ on which $U_\SA^0$ acts via the character $\chi_\lambda$, and $U_\SA^+$ acts via the augmentation $U_\SA^+\to\SA$ that sends each $e_{\alpha}^{[n]}$ to $0$ if $n>0$. 
The assumptions on $M$ (and the PBW-theorem for $U_\SA)$) imply that there is a direct sum $V$ of $r$ copies of $\Delta_\SA(\lambda)$ and a surjective homomorphism $f\colon V\to M$ that is an isomorphism on the $\lambda$-weight space. Now $L_\SK(\lambda)$ is the only non-zero quotient of $\Delta_\SK(\lambda)$ that is finite dimensional, cf.~\cite[Proposition 3.2]{L2}, hence $f_\SK$ must induce an isomorphism $L_\SK\cong M_\SK$, where $L_\SK$ is a quotient of $V_\SK$ that is isomorphic to a  direct sum of $r$ copies of $L_\SK(\lambda)$. Hence $f$ factors over a (surjective) homomorphism $\tilde f\colon W\to M$, where $W$ is a direct sum of $r$ copies of the Weyl module $W_\SA(\lambda)$. We now show that $\tilde f$ is  an isomorphism. As it is surjective, it is sufficient to show that the characters of $W$ and $M$ agree, i.e.~we need to show that the character of $M$ is $r$ times the Weyl character $\chi(\lambda)$. 

The above shows that  the character of $M_\SK$ is  $r\chi(\lambda)$. Let $\SF$ be the residue field of $\SA$. Then the character of $M_\SF$ is at least $r\chi(\lambda)$. As $M_\SF$ is generated by its $\lambda$-weight space of dimension $r$ it must be at most $r\chi(\lambda)$, cf.~the proof of Proposition 1.22 in \cite{APW}. So the character of $M_\SF$ is $r\chi(\lambda)$,  hence the characters of $M_\SF$ and of $M_\SK$ agree. This  implies that $M$ is free as an $\SA$-module with character $r\chi(\lambda)$ as well (cf.~Section 1.21 in \cite{APW}). 
\end{proof}

For $\lambda\in X$ define $M_{[\lambda]}$ as the $\lambda$-weight space in the quotient $M/N(\lambda)$, where $N(\lambda)\subset M$ denotes the $U_\SA$-submodule in $M$ that is generated by all weight spaces $M_\mu$ such that $\mu\not\le\lambda$.

\begin{proposition}\label{prop-charWf} Suppose that $\SA$ is generic.  Let $M$ be an object in $\CO_\SA$. The following statements are equivalent.
\begin{enumerate}
\item The set of weights of $M$ is finite and  $M_{[\nu]}$ is a free $\SA$-module of finite rank for all $\nu\in X$. 
\item  $M$ admits a Weyl filtration.
\end{enumerate}
If either of the above holds, then the multiplicity of $W_\SA(\mu)$ in a Weyl filtration equals the rank of $M_{[\mu]}$. In particular, $M_{[\mu]}\ne 0$ implies that $\mu$ is dominant. 
\end{proposition} 
\begin{proof}  If $M$ admits a Weyl filtration, then its set of weights is finite. Standard arguments show that  $\Ext^1_{U_\SA}(W_\SA(\lambda),W_\SA(\nu))=0$ if $\nu\not>\lambda$. Now let $\nu\in X$. Then there is a Weyl filtration  $0=M_0\subset M_1\subset\dots\subset M_n=M$ and some $1\le r\le s\le n$ such that $M_i/M_{i-1}$ has a highest weight $\not\le\nu$ if $i<r$, has highest weight $\nu$ if $r\le i\le s$, and has a highest weight $<\nu$, if $i>s$. Hence $M_s/M_r$ is a direct sum of copies of $W_\SA(\nu)$ and $M_{[\nu]}=(M_s/M_r)_{\nu}$ is free of finite rank r-s as an $\SA$-module. Hence (2) implies (1). 

Now assume that (1) holds. As the set of weights of $M$ is finite there  must be a minimal $\mu$ such that $M_{[\mu]}\ne 0$. Let us fix such a $\mu$. Consider  $N:=N(\mu)\subset M$ and the quotient $M^\prime=M/N$. Then $M^\prime_\mu=M_{[\mu]}$. We claim the following.
\begin{enumerate}
\item[(a)] We have $N_{[\nu]}\cong M_{[\nu]}$ for all $\nu\ne\mu$, and $N_{[\mu]}=0$. 
\item[(b)] $M^\prime$ is generated, as a $U^-_\SA$-module, by its $\mu$-weight space (which is a free $\SA$-module of finite rank by the above).
\end{enumerate}
If both statements are true, then we can prove that $M$ admits a Weyl filtration  as follows. From (a) we can deduce, using an inductive argument, that $N$ admits a Weyl filtration. Note that the proof of Lemma \ref{lemma-freehws} shows that the fact that the set of weights of $M^\prime$ is finite implies that $\mu$ is dominant. Then from (b) we deduce, using Lemma \ref{lemma-freehws}, that $M/N=M^\prime$ is isomorphic to a direct sum of copies of $W_\SA(\mu)$. Hence $M$ admits a Weyl filtration. 

So let us prove (a). By definition, $N$ is generated by all weight spaces $N_\nu$ with $\nu\not\le\mu$. This implies that $N_{[\nu]}=0$ for all $\nu\le\mu$. The minimality of $\mu$ implies $M_{[\nu]}=0$ for all $\nu<\mu$. Hence it remains to show that $N_{[\nu]}=M_{[\nu]}$ for all $\nu\not\le\mu$. But this is implied by $N_\nu=M_\nu$ for all such $\nu$.

We prove (b).  If $M^\prime$ wasn't generated by its $\mu$-weight space, then there would exist a weight $\nu<\mu$ such that $M^\prime_{[\nu]}\ne 0$. But this would imply $M_{[\nu]}\ne 0$, a contradiction to the minimality of $\mu$. 

Note that the inductive argument above, together with Lemma \ref{lemma-freehws}, proves the remaining statement. 
\end{proof}

\subsection{Combinatorial models for objects in $\CO_\SA^W$}
Now we want to show that the objects in $\CX_\SA$ that are free of finite rank over $\SA$, correspond, via the realization functor $\mathsf{R}$, to the objects in $\CO_\SA^W$. In a first step we are interested in what happens if we apply the functor $\mathsf{S}$ to objects that admit a Weyl filtration. 

\begin{proposition}\label{prop-equivcat} Suppose that $\SA$ is generic. 
\begin{enumerate}
\item Let $M$ be an object in $\CO_\SA^W$. Then $\mathsf{S}(M)$ is an object in $\CX_\SA$. 
\item For all dominant $\lambda$ we have $\mathsf{S}(W_\SA(\lambda))\cong S_{\emptyset,\SA}(\lambda)$.
\end{enumerate}
\end{proposition}
\begin{proof} We prove claim (1).  In view of Lemma \ref{lemma-funS} we need to  show that the property (X3) is satisfied for $\mathsf{S}(M)$. So let $\mu\in X$. As $M_\SK$ is semisimple we have  $(\mathsf{S}(M)_\mu)_\SK=(\im F_\mu)_\SK\oplus(\ker E_\mu)_\SK$. As $M$ admits a Weyl filtration, Proposition \ref{prop-Weylchar} implies that $M$ is free as an $\SA$-module. By Lemma \ref{lemma-X3}, we now only need to   show that property (X3c) holds. Let
$
\{0\}=M_0\subset M_1\subset \dots\subset M_n=M
$ be a filtration 
such that $M_{i+1}/M_i$ is isomorphic to $W_\SA(\mu_i)$. As in the proof of Proposition \ref{prop-charWf}  we can assume that   there exists an integer $r$ such that $\mu<\mu_i$ implies $i\le r$. It follows that  $\im F_\mu =(M_r)_\mu\subset M_\mu$. As the quotient $M/M_r$ admits a filtration with subquotients isomorphic to Weyl modules with dominant highest weights,  it is free as an $\SA$-module by Proposition \ref{prop-Weylchar}. In particular, its $\mu$-weight space is free. By the above, this identifies with $M_\mu/\im F_\mu$. Hence property (X3c) holds, so $\mathsf{S}(M)$ is an object in $\CX_\SA$. 

Now we prove claim (2). 
For $N=\mathsf{S}(W_\SA(\lambda))$ we have $N_\mu=\im F_\mu$ for all $\mu\ne \lambda$. Hence this is an indecomposable $\emptyset$-standard object, so $N\cong S_{\emptyset,\SA}(\nu)$ for some $\nu\in X$ by Proposition \ref{prop-catGammaStand}.  A comparison of weights shows $\nu=\lambda$.
\end{proof}

\begin{definition}\label{def-Xf} We denote by $\CX^f$ the full subcategory of $\CX$ that contains all objects $M$ that are finitely generated as $\SA$-modules. 
\end{definition}
As we assume that each weight space of an object in $\CX$ is finitely generated, the property in the definition above is equivalent to the set of weights being finite. 
The following result shows that we have obtained a model category for the category of objects in $\CO_\SA$ that admit a Weyl filtration. 
\begin{theorem} \label{thm-Wflag} Suppose that $\SA$ is generic. Then the functors $\mathsf{S}$ and $\mathsf{R}$ induce mutually inverse equivalences between the categories $\CX^{f}$ and $\CO^{W}$.
\end{theorem}

\begin{proof}  In view of Proposition \ref{prop-equivcat} we need to show that $\mathsf{S}(M)$ is finitely generated as an  $\SA$-module for all objects $M$ in $\CO$ that admit a Weyl filtration and, conversely, that $\mathsf{R}(M)$ is admits a Weyl filtration for all objects $M$ in $\CX^f$. The first statement follows easily from the facts that the functor $\mathsf{S}$ is the identity functor on the underlying $\SA$-modules and that each Weyl module is free of finite rank as an $\SA$-module. 

Now suppose that $M$ is an object in $\CX$ that is finitely generated as an $\SA$-module. We want to employ Proposition \ref{prop-charWf}, so we need to check that $\mathsf{R}(M)$ has a finite set of weights and that $\mathsf{R}(M)_{[\mu]}$ is a free $\SA$-module of finite rank for all $\mu\in X$. The first statement is clear. For the second, note that we can canonically identify  $\mathsf{R}(M)_{[\mu]}$ with $M_\mu/\im F_\mu$. The latter is, by definition of the category $\CX$, a free $\SA$-module, and of finite rank as $M$ is of finite rank.
\end{proof}
Note that we deduce, in particular, that each object in $\CX^f$ is even {\em free} of finite rank as an $\SA$-module. 
\subsection{Base change, revisited}

Now let $\SA\to\SB$ be a homomorphism of factorial $\SZ$-algebras and suppose that both $\SA$ and $\SB$ are generic. The following result strengthens Lemma \ref{lemma-basechange} in the case that $M$ is an object in $\CX^f_\SA$, but $\SA\to\SB$ is not necessarily flat.

\begin{lemma}\label{lemma-genbaschg}  Let  $M$ be an object in  $\CX_\SA^{f}$, then $M_\SB$ is an object in $\CX_\SB^{f}$.
 \end{lemma}

\begin{proof} As we have seen above $M$ is a free $\SA$-module of finite rank. Hence $M_\SB$ is a free $\SB$-module of finite rank. So we only need to show that $M_\SB$ is an object in $\CX_\SB$. But the base change functor $\otimes_\SA\SB$ commutes with the functors $\mathsf{S}$ and $\mathsf{R}$. Theorem \ref{thm-Wflag} shows that $\mathsf{R}(M)$ admits a Weyl filtration. For the Weyl modules we  have  $W_\SA(\lambda)_{\SB}\cong W_{\SB}(\lambda)$ since both are generated by their $\lambda$-weight space and the characters agree (cf.~Proposition \ref{prop-Weylchar}). Hence $\mathsf{R}(M)\otimes_\SA\SB$ admits a Weyl filtration, and again using Theorem  \ref{thm-Wflag}  allows us to deduce that $\mathsf{S}(\mathsf{R}(M)\otimes_\SA\SB)=M_\SB$ is contained in $\CX$. \end{proof}

\subsection{The relevance of cyclotomic polynomials}
Now we show that the relevant torsion spaces are annihiliated by a product of quantum integers.  
\begin{lemma} \label{lemma-Nsatmax}  Let $M$ be an object in $\CX^f$. For all $\mu\in X$, the torsion part of the $\SA$-module   $M_{\delta\mu}/M_{\{\mu\}}$ is annihilated by (the image in $\SA$ of) a product of quantum integers $[n]_d$. 
\end{lemma}
\begin{proof} Let $\mu$ be an arbitrary weight of $M$.  As $M_{\delta\mu}$ is finitely generated it suffices to prove the following.
Let $m\in M_{\delta\mu}$ and $\xi\in{\SA}\setminus\{0\}$ be such that the following holds. 
\begin{enumerate}
\item $\xi$ is irreducible in $\SA$.
\item $\xi m \in M_{\{\mu\}}$, but $m\not\in M_{\{\mu\}}$.
\end{enumerate}
Then $\xi$ divides $[n]_d$ for some $n,d>0$. 

Consider the $\SA$-algebra $\SB:={\SA}/\xi {\SA}$. As $\xi $ is irreducible, $\SB$ is a factorial domain. As $\xi m\in M_{\{\mu\}}$, there exists an element $\tilde m\in F_\mu(M_{\delta\mu})$ such that $E_\mu(\tilde m)=\xi m$. As $E_\mu$ is injective when restricted to $F_\mu(M_{\delta\mu})$ and as $m\not\in M_{\{\mu\}}$, $\tilde m$ is not divisible by $\xi $ in $F_\mu(M_{\delta\mu})$. The image $\tilde m^\prime$ of $\tilde m$ in $(M_\mu)_{\SB}$ is hence non-zero. But since $E_\mu(\tilde m)$ is divisible by $\xi$, we have  $E^{M_\SB}_\mu(\tilde m^\prime)=0$. As $\tilde m^\prime$ is contained in the image of $F^{M_\SB}_\mu$, we deduce that $E^{M_\SB}_\mu$ is not injective when restricted to this image. Now Lemma \ref{lemma-genbaschg} implies that  $\SB$ is not generic. Hence $[n]_d$ must vanish in $\SB$ for some pair $(n,d)$, so $\xi$ must divide $[n]_d$.
\end{proof}

In particular, $S_{\SA.\Gamma}(\lambda)$ depends only on the irreducible factors of the quantum integers that are contained in $\Gamma$.

\subsection{$\Gamma$-Standard objects in $\CX^f$}
We assume that $\SA$ is local and generic. 
 The next main result is that for a dominant weight $\lambda$ and any subset $\Gamma$ of $\SA^{\irr}$  the object $S_\Gamma(\lambda)$ is  contained in $\CX^f$. Hence it admits a Weyl filtration when considered as a $U_\SA$-module. Denote by $\SF$ the residue field of $\SA$.   For an $\SA$-module $M$ we denote by $\ol M:=M\otimes_{\SA}\SF$ the corresponding $\SF$-vector space, and for an $\SA$-linear homomorphism $f\colon M\to N$ we denote by $\ol f\colon \ol M\to \ol N$ the induced $\SF$-linear homomorphism.

\begin{lemma}\label{lemma-minmax}  Let $M$ be an object in $\CX^f$. Suppose that $\mu\in X$ is such that $M_{\delta\mu}/M_{\{\mu\}}$ has a non-vanishing torsion part.  Then $\mu$ is dominant. 
\end{lemma}

\begin{proof} Recall that $M_{\{\mu\}}=E_\mu(F_\mu(M_{\delta\mu}))\subset M_{\delta\mu}$.  Hence  our assumption implies that  the restriction of $\ol E_\mu\colon \ol{M_\mu}\to \ol{M_{\delta\mu}}$ to $\ol{F_\mu(M_{\delta\mu})}$ has a non-trivial kernel.  This means that the $U_\SF$-module $\ol{\mathsf{R}(M)}$ contains a non-trivial primitive vector of weight $\mu$. But $\mathsf{R}(M)$  admits a Weyl filtration by Theorem \ref{thm-Wflag}. Hence there must be a Weyl module $\ol{W_\SA(\lambda)}$ for $U_{\SF}$ containing a non-trivial primitive vector of weight $\mu$. This implies that $\mu$ is dominant. \end{proof}

\begin{proposition}\label{prop-Sdom} Let $\lambda$ be a dominant weight and $\Gamma$ a subset of $\SA^{\irr}$. Then  $S_\Gamma(\lambda)$ is an object in $\CX^f$.  In particular, $\mathsf{R}(S_\Gamma(\lambda))$ admits a Weyl filtration.
\end{proposition}
\begin{proof} We need to show that the set of weights of $S_\Gamma(\lambda)$ is finite, cf. the remark after Definition \ref{def-Xf}. Denote by $M\subset S_\Gamma(\lambda)$ the smallest $E$- and $F$-stable $X$-graded subspace of $S_\Gamma(\lambda)$ that contains $S_\Gamma(\lambda)_\mu$ for all dominant weights $\mu$.  If $M\ne S_\Gamma(\lambda)$, then there exists a weight $\mu$ which is not dominant such that $M_\mu\ne S_\Gamma(\lambda)_\mu$. If we take a maximal $\mu$ with this property, we deduce that $S_\Gamma(\lambda)_{\delta\mu}=M_{\delta\mu}$ and hence $\im F_\mu^M=\im F_\mu^{S_\Gamma(\lambda)}$. From this we obtain $M_{\{\mu\}}=S_\Gamma(\lambda)_{\{\mu\}}$. But from Lemma \ref{lemma-minmax} and the fact that $\mu$ is not dominant we deduce that $S_{\Gamma}(\lambda)_{\{\mu\}}=S_{\Gamma}(\lambda)_{\{\mu\},\max}$, so the construction of $S_\Gamma(\lambda)$ implies $S_\Gamma(\lambda)_\mu=\im F_\mu^{S_\Gamma(\lambda)}$. As we have seen above, this equals $\im F_\mu^M$, hence $M_\mu=S_\Gamma(\lambda)_\mu$ which contradicts our choice of $\mu$. 

Hence $S_\Gamma(\lambda)$ is ``generated by dominant weights''. Hence so is $S_\Gamma(\lambda)\otimes_\SA\SK$, where $\SK$ is the quotient field of $\SA$. Now $S_\Gamma(\lambda)\otimes_\SA\SK$ splits into a direct sum of Weyl modules (when considered as a representation of $U_\SK$), and these Weyl modules have dominant weights as highest weights. Hence the set of weights is finite. The last statement now follows from Theorem \ref{thm-Wflag}.
\end{proof}

\section{Maximal extensions and tilting modules}
Let $\lambda$ be an element in $X$ and $\Gamma$ subset of $\SA^{\irr}$. Then we have the $\Gamma$-standard object $S_{\Gamma}(\lambda)$ that is characterized in Proposition \ref{prop-catGammaStand}. Theorem \ref{thm-XU} allows us to view this as a module for the quantum group $U_\SA$, and even  as an object in $\CO_{\SA}$. A natural question  is if we can characterize this object as a $U_\SA$-module. For dominant weights $\lambda$ and $\Gamma=\emptyset$ this is answered  in Proposition \ref{prop-equivcat}: the object $S_\emptyset(\lambda)$ corresponds to the Weyl module with highest weight $\lambda$. In this section we give the answer in the other extreme case. For dominant $\lambda$ and $\Gamma=\SA^{\irr}$ the object $S_{\Gamma\setminus\{0\}}(\lambda)$ corresponds to the indecomposable tilting module with highest weight $\lambda$.

\subsection{Contravariant forms on objects in $\CX$}

Let $M$ be an object in $\CX_\SA$ and let $b\colon M\times M\to \SA$ be an $\SA$-bilinear form.
\begin{definition} We say that $b$ is a {\em symmetric contravariant form on $M$} if the following holds.
\begin{enumerate}
\item  $b$ is symmetric.
\item $b(m,n)=0$ if $m\in M_\mu$ and $n\in M_\nu$ and $\mu\ne\nu$. 
\item  For all $\mu\in X$, $\alpha\in\Pi$, $n>0$, $x\in M_{\mu}$, $y\in M_{\mu+n\alpha}$ we have
$$
b(E_{\mu,\alpha,n}(x),y)=b(x,F_{\mu,\alpha,n}(y)).
$$
\end{enumerate}
\end{definition}

In order to study contravariant forms, the following quite general result will be helpful for us. 

\begin{lemma}\label{lemma-biforms} Let $S$ and $T$ be $\SA$-modules and assume that $T$ is projective as an $\SA$-module.  Let $F\colon S\to T$ and $E\colon T\to S$ be homomorphisms. Suppose that 
 $b_S\colon S\times S\to \SA$ and $b_T\colon T\times T\to \SA$ are symmetric, non-degenerate bilinear forms such that  
$$
b_T(F(s),t)=b_S(s,E(t))
$$
for all $s\in S$ and $t\in T$. 
If the inclusion $i_F\colon F(S)\subset T$ splits, then  the inclusion $i_E\colon E(T)\subset S$ splits as well.
\end{lemma}

\begin{proof} We denote by $O^\ast=\Hom_\SA(O,\SA)$ the $\SA$-linear dual of an $\SA$-module $O$. From the non-degeneracy of $b_S$ and the adjointness property it follows that $\ker F$ coincides with the set of all $s\in S$ such that $b_S(s,s^\prime)=0$ for all $s^\prime\in E(T)$. Hence we can define a homomorphism $\phi\colon E(T)\to F(S)^\ast$ by setting $\phi(E(t))(F(s))=b_S(s,E(t))$. Then the right hand side of the diagram

\centerline{
\xymatrix{
T\ar[d]_{b_T}\ar[r]^E&E(T)\ar[r]^{i_E}\ar[d]^{\phi}&S\ar[d]^{b_S}\\
T^\ast\ar[r]^{i_F^\ast}&F(S)^\ast\ar[r]^{F^\ast}&S^\ast
}
}
\noindent
commutes. By the adjointness property, $\phi(E(t))(F(s))=b_T(F(s),t)$ for all $s\in S$ and $t\in T$, hence also the left hand side commutes. The vertical homomorphisms on the left and the right are isomorphisms.
 As $i_E$ is injective, $\phi$ is injective. Suppose that   $i_F\colon F(S)\subset T$ splits. Then the dual homomorphism   $i_F^\ast\colon T^\ast\to F(S)^\ast$ is surjective. Hence $\phi$ is surjective, so it is an isomorphism. As $F(S)$ is projective (it is a direct summand of $T$), the surjective homomorphism $S\to F(S)$ splits, and hence $F^\ast\colon F(S)^\ast\to S^\ast$ splits. Hence the inclusion $i_E\colon E(T)\to S$ splits. \end{proof}

The following is our main application of the previous lemma. 
\begin{proposition}\label{prop-cform}  Let $M$ be an object in $\CX_\SA$. Suppose that there exists a non-degenerate symmetric contravariant form on $M$ and that $M_\mu$ is a free $\SA$-module for all $\mu\in X$.  Then $M$ is $\SA^{\irr}$-saturated, i.e.~the quotient $M_{\delta\mu}/M_{\lgl\mu\rgl}$ is torsion free  for all $\mu\in X$.  
\end{proposition}

\begin{proof}
Let   $b\colon M\times M\to \SA$ be a non-degenerate symmetric contravariant form on $M$. For all $\mu\in X$ it  induces  symmetric, non-degenerate bilinear forms $b_{\delta\mu}$ and $b_\mu$ on the (free) $\SA$-modules $M_{\delta\mu}$ and $M_\mu$, resp., with
$$
b_\mu(F_\mu(v),w)=b_{\delta\mu}(v,E_\mu(w))
$$
for all $v\in M_{\delta\mu}$ and $w\in M_\mu$. Moreover, since $M$ is an object in $\CX_\SA$, the quotient $M_\mu/F_\mu(M_{\delta\mu})$ is a free $\SA$-module. Hence the inclusion $F(M_{\delta\mu})\subset M_\mu$ splits. Hence we can apply Lemma \ref{lemma-biforms} and deduce that the inclusion $M_{\lgl\mu\rgl}=E_\mu(M_\mu)\subset M_{\delta\mu}$ splits. As $M_{\delta\mu}$ is free, this implies that the quotient   $M_{\delta\mu}/M_{\lgl\mu\rgl}$ is  torsion free.  It follows that $M_{\{\mu\},\max}\subset M_{\lgl\mu\rgl}$, hence  $M_{\{\mu\},\max}=M_{\lgl\mu\rgl}$.
\end{proof}

\subsection{Tilting modules} There is also the notion of a contravariant form on representations of  a quantum group. Before we come to its definition, recall that there is an  antiautomorphism $\tau$ of order $2$ on $U_\SA$ that maps $e_\alpha$ to $f_\alpha$ and $k_\alpha^{\pm 1}$ to $k_\alpha^{\pm 1}$ (this is an immediate consequence of the definition of $U_\SZ$ by generators and relations). The contravariant dual of an object $M$ in $\CO_\SA$ is given by $dM=\bigoplus_{\nu\in X} M_\nu^\ast \subset M^\ast$  with the action of $U_\SA$ twisted by the antiautomorphism $\tau$. A homomorphism $M\to dM$ is hence the same as an $\SA$-bilinear form $b\colon M\times M\to \SA$ that satisfies
\begin{itemize}
\item $b(x.m,n)=b(m,\tau(x).n)$ for all $x\in U_\SA$, $m,n\in M$.
\item $b(m,n)=0$ if $m\in M_\mu$, $n\in M_\nu$ and $\mu\ne\nu$.
\end{itemize}
Such a form is also called a {\em contravariant form on $M$}. 

Assume now that $\SA$ is local. 
For a dominant weight  $\lambda\in X$ we denote by $T(\lambda)=T_\SA(\lambda)$ the indecomposable tilting module for $U_\SA$ associated with the weight $\lambda$. 
\begin{lemma}\label{lemma-exsym} We assume that $\SA$ is local. Suppose that $2$ is invertible in $\SA$. Then there exists a symmetric, non-degenerate contravariant form on $T(\lambda)$. 
\end{lemma}
\begin{proof} Each tilting module is selfdual with respect to the contravariant duality. There exists hence a non-degenerate contravariant bilinear form $b^\prime$ on $T(\lambda)$. We set $b:=b^\prime+t.b^\prime$, where $t.b^\prime$ is obtained from $b^\prime$ by switching the places of the arguments. So $b$ is now a symmetric contravariant bilinear form. As the highest weight space $T(\lambda)_\lambda$ is a free $\SA$-module of rank $1$ and as $2$ is invertible in $\SA$, we deduce that the restriction of $b$ to $T(\lambda)_\lambda$ is non-degenerate. So $b$ induces a homomorphism $T(\lambda)\to dT(\lambda)\cong T(\lambda)$ that is an isomorphism on the $\lambda$-weight space. But then this homomorphism must be an isomorphism. So $b$ is non-degenerate.
\end{proof}

\begin{theorem} \label{thm-tiltmax} Suppose that $\SA$ is generic and local and that $2$ is invertible in $\SA$. Let $\lambda$ be a dominant element in $X$. Then $\mathsf{S}(T(\lambda))\cong S_{\SA^{\irr}}(\lambda)$ and hence $\mathsf{R}(S_{\SA^{\irr}}(\lambda))\cong T(\lambda)$. 
\end{theorem}
\begin{proof} As $T(\lambda)$ admits a Weyl filtration with subquotients having dominant highest weights, it is an object in $\CO_\SA^W$. So we can apply Proposition \ref{prop-equivcat} and deduce that   $T:=\mathsf{S}(T(\lambda))$ is contained in $\CX_\SA$.  By Lemma \ref{lemma-exsym} there exists a symmetric, non-degenerate contravariant bilinear form on $T(\lambda)$. Such a bilinear form induces a non-degenerate  symmetric contravariant bilinear form on $T$.  Proposition \ref{prop-cform} implies that $T$ is $\SA^{\irr}$-saturated.  As it is indecomposable, part (3) of Proposition  \ref{prop-catGammaStand}  implies that we have $T\cong S_{\SA^{\irr}}(\mu)$ for some $\mu$. As $\lambda$ is the maximal weight of $T(\lambda)$, we have $\mu=\lambda$. 
\end{proof}

\begin{remark} As explained in Remark \ref{rem-algo}, the proof of Proposition \ref{prop-GammaExtObj} can be read as an algorithm for constructing the objects $S_{\SA^{\irr}}(\lambda)$. Using the above theorem we hence obtain an algorithm that produces the weight spaces of the tilting modules $T(\lambda)$ for dominant $\lambda$.
\end{remark}

\subsection{Proof of the Main Theorem} We now are ready to collect the results in the previous sections. Suppose that $p$ is an odd prime and $p\ne 3$ if $R$ contains a component of type $G_2$. Let  $\SA=\SZ_\fp$ be the algebra of quantum $p$-adic integers (cf. Section \ref{subsec-groundrings}). For $l\in\DN$ denote by $\tau_l\in\SZ_{\fp}$ the $p^l$-cyclotomic polynomial. Note that this is irreducible in $\SZ_\fp$. 
Let $\Theta$ be a subset of $\DN$ and let $\Gamma_\Theta=\{\tau_l\mid l\in\Theta\}$. Then we  set for all  $\lambda\in X$ 
$$
T_{\Theta}(\lambda)=\mathsf{R}(S_{\Gamma_\Theta}(\lambda)).
$$
This is an object in $\CO_\SA$. Let us state the Main Theorem from the introduction again, and then prove it. 

\begin{MainTheorem}  Let $\lambda$ be a dominant weight.
\begin{enumerate}
\item $T_{\emptyset}(\lambda)$ is the Weyl module with highest weight $\lambda$. 
\item $T_{\DN}(\lambda)$ is the indecomposable tilting module with highest weight $\lambda$.
\item Each $T_{\Theta}(\lambda)$ admits a Weyl filtration.
\item  If $l\in\Theta$, then the character of $T_{\Theta}(\lambda)$ is a sum of tilting characters of the quantum group at an $p^l$-th root of unity.
\end{enumerate}
\end{MainTheorem}
Note that the characters of the tilting modules for a quantum group at a root of unity are known (cf. \cite{Soe}). 

\begin{proof} Part (1) is the inverse statement of part (2) in Proposition \ref{prop-equivcat}. Note that each quantum integer is a product of cyclotomic polynomials. Hence Lemma \ref{lemma-Nsatmax} implies that $S_{\Gamma_\DN}(\lambda)=S_{\SA^{\irr}}(\lambda)$. Part (2) above now follows from  Theorem \ref{thm-tiltmax}. Part (3) is Proposition \ref{prop-Sdom}. Now we prove part (4). Let $\SA=(\SZ_{\fp})_{(\tau_l)}$ be the localization of $\SZ_\fp$ at the ideal generated by $\tau_l$. Then the residue field of $\SA$ is $\DQ[\zeta_{p^l}]$, the $p^l$-cyclotomic field. From  the base change result in Lemma \ref{lemma-basechangeS} we obtain that $T:=T_{\Gamma}(\lambda)\otimes_{\SZ_\fp}\SA$ is an object in $\CX_{\SA}$. Since $l\in\Theta$ and since $\tau_{l}$ is, up to units, the only irreducible element in $\SA$, the object $\mathsf{S}(T)$ must be $\SA^{\irr}$-standard. Theorem \ref{thm-tiltmax} shows that $T$ is a tilting module in $\CO_\SA$. In particular, the specialization $T\otimes_\SA\DQ[\zeta_{p^l}]$ is a tilting module for the quantum group at a $p^l$-root of unity. 
\end{proof}

\end{document}